\documentclass[11pt,a4paper]{article}
\usepackage[T1]{fontenc}
\usepackage{lmodern}
\usepackage[margin=26mm]{geometry}
\usepackage{amsmath,amssymb,amsthm,mathtools}
\usepackage{graphicx,booktabs,array}
\usepackage{microtype}
\usepackage{enumitem}
\usepackage[font=small,labelfont=bf]{caption}
\usepackage[hidelinks]{hyperref}
\usepackage{bookmark}
\usepackage[section]{placeins}
\usepackage{flafter,float}
\usepackage{needspace}
\numberwithin{equation}{section}
\newtheorem{theorem}{Theorem}[section]
\newtheorem{lemma}[theorem]{Lemma}
\newtheorem{proposition}[theorem]{Proposition}
\newtheorem{corollary}[theorem]{Corollary}
\theoremstyle{remark}

\newcommand{\R}{\mathbb R}
\newcommand{\N}{\mathbb N}
\newcommand{\Z}{\mathbb Z}
\newcommand{\T}{\mathbb T}
\newcommand{\EE}{\mathbb E}
\newcommand{\PP}{\mathbb P}
\newcommand{\dd}{\,\mathrm d}
\newcommand{\one}{\mathbf 1}

\newcommand{\norm}[1]{\lVert#1\rVert}
\newcommand{\TV}{\mathrm{TV}}
\DeclareMathOperator{\Var}{Var}
\DeclareMathOperator{\Cov}{Cov}
\DeclareMathOperator{\diag}{diag}
\DeclareMathOperator*{\argmax}{arg\,max}
\DeclareMathOperator{\Law}{Law}
\DeclareMathOperator{\med}{med}
\DeclareMathOperator{\diam}{diam}

\setlist[enumerate]{label=\textnormal{(\roman*)},leftmargin=2.2em,itemsep=2pt,topsep=4pt}
\allowdisplaybreaks[1]
\hypersetup{pdftitle={Single-graph inference for fractal Gaussian networks},
 pdfsubject={Gaussian multiplicative chaos, single-network inference, conservative Monte Carlo calibration}}
\title{Single-graph inference for fractal Gaussian networks}
\author{Chunhao Cai\\
\small School of Mathematics (Zhuhai), Sun Yat-Sen University\\
\small \texttt{caichh9@mail.sysu.edu.cn}}
\date{}

\begin{document}
\maketitle

\begin{abstract}
We study inference on the strength of Gaussian multiplicative chaos from one
geometric graph with unobserved vertex positions.  In the planar model, the
edge-count statistic has a piecewise deterministic limit with a transition at
$\gamma=1$, whereas degree quantiles consistently estimate $\nu=\gamma^2/2$
throughout the subcritical range.  Fractional moments give uniform
finite-sample risk bounds.  At finite resolution, common graph envelopes
calibrate tests over continuous parameter boxes with unknown Poisson
intensity.  They give simultaneous coverage under adaptive refinement for a
fixed statistic, and conditional coverage for an independently piloted
size--degree residual on a fixed partition.  For the exact periodic
FFT--pixel approximation, we prove explicit total-variation rates for the
annealed spatial Cox law and the observed graph law, uniformly on every
compact subcritical range $0\le\gamma\le\Gamma<2$.  The rates allow intensity
to grow with resolution and imply uniform asymptotic coverage after
recalibration at each resolution.  Numerical studies examine parameter
refinement, residual calibration, and multiresolution stability.
\end{abstract}



\section{Introduction}
\label{sec:intro}

Fractal Gaussian networks use Gaussian multiplicative chaos (GMC) as a random
sampling measure and connect sampled vertices according to their latent
Euclidean distances.  Ghosh, Balasubramanian and Yang introduced the model,
computed mean subgraph counts, and proposed inference from a single adjacency
graph \cite{FGN}.  A single graph, however, samples one random environment, so
its typical statistics need not follow exponents obtained after averaging over
environments, as is familiar in multifractal analysis
\cite{Angeletti,Bertacco}.  Related recovery problems for ordinary random
geometric graphs are studied in \cite{Atamanchuk}.

We determine the single-environment limits of edge and degree statistics and
construct finite-sample Monte Carlo calibration over continuous parameter
boxes using common graph envelopes \cite{Dufour,GlazerStark,JouneauTorres}.
We also calibrate a residual score that uses the observed vertex count, in the
spirit of separating score construction from calibration in simulator-based
inference \cite{LF2I}.  Finally, we quantify the passage from the periodic
FFT--pixel simulator to its continuous target.  General GMC approximation
theory gives qualitative identification \cite{Shamov,JunnilaSaksman,Berestycki};
spectral, pixel, and discrete torus constructions are treated in
\cite{DelloSchiavoManifold,DelloSchiavoTorus}.  For the even-grid approximation
used here, we obtain total-variation rates for the annealed Cox and observed
graph laws and transfer them to calibrated confidence sets.  Figure~\ref{fig:observation}
shows the observation: one finite graph, including isolated vertices, with
the field and positions unobserved.

\subsection{The observation and the parameter}
\label{subsec:model}

Let $D=(0,1)^2$, and let $X$ be a centered log-correlated generalized
Gaussian field whose covariance on $D$ is
\begin{equation}
 \begin{aligned}
 K(x,y)&=\log\frac1{|x-y|}+g_0(x,y),\qquad
 g_0\in C(\overline D\times\overline D),\qquad \norm{g_0}_\infty\le B_0,\\
 K(x,y)&\ge0,\qquad x\ne y.
 \end{aligned}
 \label{eq:covariance}
\end{equation}
Assume that $X$ extends to a log-correlated field on a neighborhood of
$\overline D$ with continuous bounded covariance remainder.  For a standard
regularization $X_\varepsilon$ and $0\le\gamma<2$, use full-variance
normalization
\begin{equation}
 \mu_\gamma(\dd x)
 =\lim_{\varepsilon\downarrow0}
   e^{\gamma X_\varepsilon(x)
      -\gamma^2\EE X_\varepsilon(x)^2/2}\dd x,
 \qquad M_\gamma=\mu_\gamma(D),\qquad \nu=\frac{\gamma^2}{2},
 \label{eq:gmc}
\end{equation}
where the limit is in probability in the weak topology of finite measures.
The total mass is retained, so $\EE\mu_\gamma(\dd x)=\dd x$,
$\EE M_\gamma=1$, and $\mu_0=\dd x$.  All logarithms are natural, and
$\log_+x=\max(0,\log x)$.

For a deterministic intensity scale $n\ge1$ and a known $0<c\le1$, sample a
Poisson point process of intensity $n\mu_\gamma$.  Conditional on its positions,
connect each unordered pair independently with probability
\begin{equation}
 k_{r_n}(x,y)=e^{-|x-y|^2/r_n^2},\qquad r_n=cn^{-1/2}.
 \label{eq:kernel}
\end{equation}
Let $G_n$ be the resulting unlabelled finite graph, with isolated vertices
retained; write $N_n$, $E_n$, and $d_v$ for its vertex count, edge count, and
vertex degrees.  Conditional on the environment,
\begin{equation}
 N_n\mid\mu_\gamma\sim\operatorname{Poisson}(nM_\gamma).
 \label{eq:random-size}
\end{equation}
Thus the deterministic model scale $n$ and the observable size $N_n$ are
different quantities.  The parameter of interest is $\nu$; the covariance and
connection rule are specified.  Write $\PP_{\nu,n}$ and $\EE_{\nu,n}$ for
probability and expectation over the environment, vertices, and edge marks.

\begin{figure}[!htbp]
 \centering
 \includegraphics[width=\textwidth]{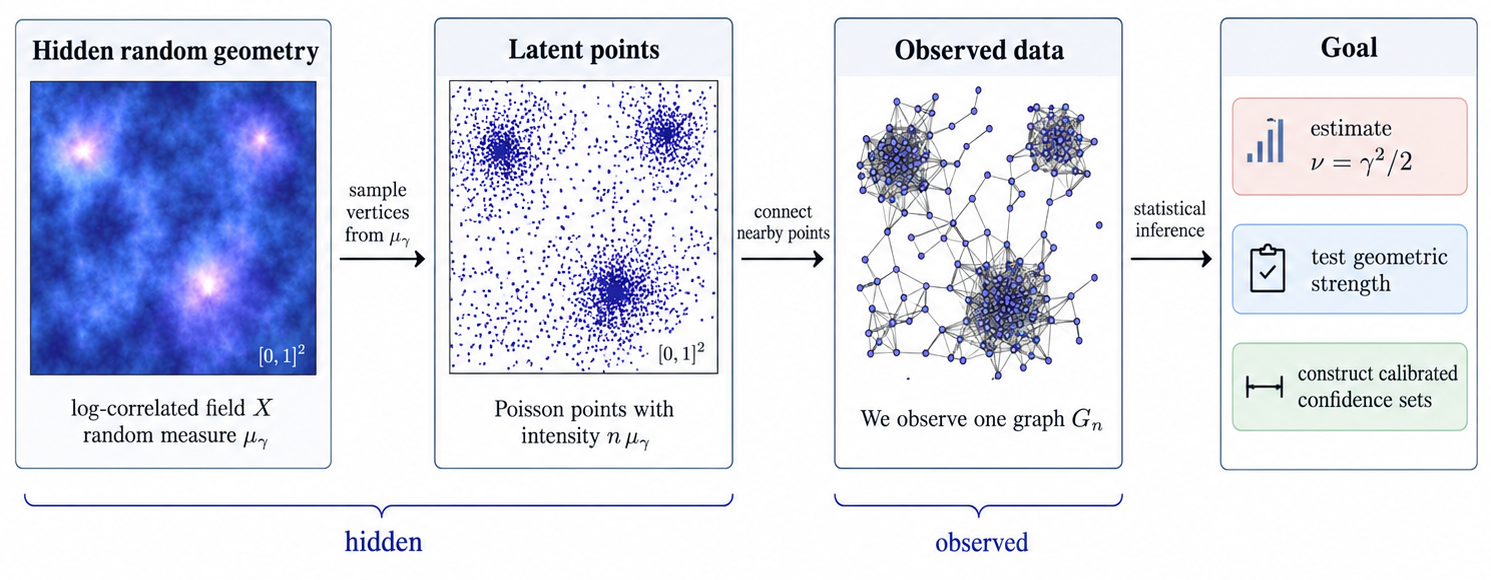}
 \caption{The observation protocol.  A hidden log-correlated field generates
 a GMC measure, vertices are sampled from that measure, and distance-dependent
 edges produce one observed network.  Inference uses only its adjacency
 structure and vertex count.  The displayed network layout is an illustration,
 not observed spatial coordinates.  The left-hand image depicts a coarse
 approximation to the random measure; it is not an ordinary density of the
 limiting GMC.}
 \label{fig:observation}
\end{figure}
\FloatBarrier

\subsection{What a single graph estimates}
\label{subsec:statistics}

For a finite graph $G$, write $N=N(G)$ and $E=E(G)$.  If
$d_{(1)}\le\cdots\le d_{(N)}$ are its ordered degrees, set
\begin{equation}
 \med_v d_v=
 \begin{cases}
 d_{(\lceil N/2\rceil)},&N\ge1,\\
 0,&N=0,
 \end{cases}
 \qquad
 T_D(G)=\frac{2\log(1+\med_v d_v)}{\log(N\vee2)},\qquad
 T_E^0(G)=\frac{\log(E\vee1)}{\log(N\vee2)}-1.
 \label{eq:statistics}
\end{equation}
Then $0\le T_D\le2$.  For $0<\bar\nu<2$, define
$\widehat\nu_n=\Pi_{[0,\bar\nu]}T_D(G_n)$, where $\Pi$ denotes interval
projection.
The limiting map for the edge statistic is
\begin{equation}
 b(\nu)=
 \begin{cases}
  \nu,&0\le\nu\le\tfrac12,\\
  2\sqrt{2\nu}-\nu-1,&\tfrac12<\nu<2.
 \end{cases}
 \label{eq:edge-map}
\end{equation}
The limits below are conditional on one fixed environment; the tail and risk
bounds are under the joint law $\PP_{\nu,n}$.

\begin{theorem}
\label{thm:single}
For each $0\le\gamma<2$, for almost every realization of $\mu_\gamma$, as $n\to\infty$,
\begin{equation}
 T_E^0(G_n)\longrightarrow b(\nu),\qquad
 T_D(G_n)\longrightarrow\nu
 \quad\text{in conditional probability}.
 \label{eq:main-limits}
\end{equation}
For every $0<\bar\nu<2$, there are finite positive constants
$A_1,A_2,A_3,\epsilon_0$, depending only on $K,c,\bar\nu$, such that
for $n\ge e$ and $0<\epsilon\le\epsilon_0$,
\begin{align}
 \sup_{0\le\nu\le\bar\nu}
 \PP_{\nu,n}(|\widehat\nu_n-\nu|>\epsilon)
 &\le A_1e^{-\epsilon^2\log n/(64\bar\nu)}+\frac{A_2}{n},
 \label{eq:main-tail}\\
 \shortintertext{\textit{and}}
 \sup_{0\le\nu\le\bar\nu}
 \EE_{\nu,n}|\widehat\nu_n-\nu|^2
 &\le\frac{A_3}{\log n}.
 \label{eq:main-risk}
\end{align}
\end{theorem}

On $N\ge2$ and $E\ge1$, $T_E^0$ equals the single-pass statistic proposed in
\cite[Section~4.8, equation~(15), author version]{FGN}.  Its motivation is the
mean edge-count exponent $1+\nu$ for $\nu<1$
\cite[Theorem~4.3, author version]{FGN}.  In contrast, \eqref{eq:main-limits} gives the typical
single-environment exponent $1+b(\nu)$.  For $1<\gamma<\sqrt2$,
\begin{equation}
 \nu-b(\nu)=(\gamma-1)^2.
 \label{eq:bias-gap}
\end{equation}
The degree statistic reads the mass-typical local dimension, whereas the edge
count reads a second-order mass statistic.  This distinction also separates a
$\mu_\gamma$-sampled vertex from the deterministic inserted location in
\cite[Theorem~3.1, author version]{FGN}.

\subsection{Calibration over continuous parameter sets}
\label{subsec:calibration-main}

Let $\theta=(\gamma,n)$ range over a fixed rectangle
\begin{equation}
 \Theta=[0,\Gamma]\times[n_-,n_+],\qquad
 0\le\Gamma<2,\qquad 1\le n_-<n_+<\infty.
 \label{eq:parameter-domain}
\end{equation}
A finite-resolution simulator uses a finite partition $(A_i)_i$ of the
observation domain and a centered Gaussian vector $Y=(Y_i)_i$ with known
covariance.  If $v_i=\EE Y_i^2$, its density is
\begin{equation}
 w_{\gamma,h}(x)=e^{\gamma Y_i-\gamma^2v_i/2},\qquad x\in A_i.
 \label{eq:pixel-density}
\end{equation}
Here $h$ indexes spatial resolution.  Its graph law is $Q_{\theta,h}$, obtained
from intensity $nw_{\gamma,h}\dd x$ and the specified connection kernel.
The construction also applies to a periodic domain with its periodic distance.

Fix one statistic $T$, either $T_D$ or
$T_E=\Pi_{[0,1]}T_E^0$.  For each of $R$ independent simulator environments,
use common Poisson and edge marks for all $\theta\in\Theta$.  For every
parameter rectangle $B\subset\Theta$, Section~\ref{sec:calibration} constructs
simultaneous bounds
\begin{equation}
 L_{B,r}\le T(G_\theta^{(r)})\le U_{B,r},
 \qquad \theta\in B,\quad 1\le r\le R.
 \label{eq:envelope-intro}
\end{equation}
For the observed value $t=T(G)$, set
\begin{align}
 p_B^+(t)&=\frac{1+\sum_{r=1}^R\one\{U_{B,r}\ge t\}}{R+1},&
 p_B^-(t)&=\frac{1+\sum_{r=1}^R\one\{L_{B,r}\le t\}}{R+1},
 \label{eq:box-pvalues}\\
 \shortintertext{and}
 p_B(t)&=\min\{1,2\min(p_B^+(t),p_B^-(t))\}.&&
 \label{eq:box-twosided}
\end{align}
Fix $0<\alpha<1$.  Starting from a fixed countable rectangle-refinement
tree, a measurable rule may select the next rectangle using the data and all
previous calculations.  A rectangle is discarded only if $p_B(t)\le\alpha$;
let $\Theta_j$ be the union of all undiscarded rectangles after step $j$.  The same
$R$ coupled replicas are reused throughout.

For parameter localization, we also use the two-dimensional observation
\begin{equation}
 S(G)=\bigl(\log(N\vee2),T_D(G)\bigr)^\top.
 \label{eq:joint-observation}
\end{equation}
Let $\mathcal B$ be a fixed finite rectangle partition of $\Theta$.  An
independent pilot sample $\mathcal D_0$ supplies $m_B\in\R^2$ and a positive
definite matrix $\Sigma_B$ for each $B\in\mathcal B$.  Put
\begin{align}
 \xi_B&=\frac{(\Sigma_B)_{21}}{(\Sigma_B)_{11}},&
 \sigma_B^2&=(\Sigma_B)_{22}
             -\frac{(\Sigma_B)_{21}^2}{(\Sigma_B)_{11}}>0,
 \label{eq:residual-coefficients}\\
 \shortintertext{and}
 a_B&=\frac{(-\xi_B,1)^\top}{\sigma_B},&
 \rho_B(G)&=\bigl|a_B^\top(S(G)-m_B)\bigr|.
 \label{eq:residual-score}
\end{align}

For the independent calibration replicas, the graph envelopes give
$\ell_{B,r}\le S(G_\theta^{(r)})\le u_{B,r}$ coordinatewise for $\theta\in B$.
Let
\begin{equation}
 U^\rho_{B,r}=\max_{z\in[\ell_{B,r},u_{B,r}]}
                 |a_B^\top(z-m_B)|,\qquad
 p_B^\rho(G)=\frac{1+\sum_{r=1}^R\one\{U^\rho_{B,r}\ge\rho_B(G)\}}{R+1}.
 \label{eq:residual-calibration}
\end{equation}
The residual confidence set is
\begin{equation}
 C_\alpha^\rho(G)=
 \bigcup_{\substack{B\in\mathcal B\,:\,p_B^\rho(G)>\alpha}}
       \{\gamma^2/2:(\gamma,n)\in B\}.
 \label{eq:residual-confidence-set}
\end{equation}
Assume that the pilot, calibration sample, and target graph are independent,
and that $\mathcal B$ and the pilot scores are fixed before calibration.  Let
$P_h^{\rm aux}$ be the law of all auxiliary randomness at resolution $h$: the
$R$ calibration environments, their common Poisson and edge marks, and, for the
residual procedure, the independent pilot $\mathcal D_0$.

\Needspace{12\baselineskip}
\begin{theorem}
\label{thm:calibration}
\begin{enumerate}
\item For a fixed statistic and the common-sample refinement procedure,
\begin{equation}
 \sup_{\theta\in\Theta}
 (Q_{\theta,h}\otimes P_h^{\rm aux})
 \{\exists j\ge0:\theta\notin\Theta_j\}\le\alpha.
 \label{eq:main-calibration}
\end{equation}
The retained sets satisfy
\[
 \Theta_{j+1}\subseteq \Theta_j,\qquad j\ge0.
\]
For each fixed $\theta$ and every nested sequence $B_k\ni\theta$ with
$B_{k+1}\subseteq B_k$ and $\diam B_k\to0$,
\[
 P_h^{\rm aux}\!\left\{\exists k_0<\infty:\
 L_{B_k,r}=U_{B_k,r}=T(G_\theta^{(r)})
 \ \forall k\ge k_0,\ 1\le r\le R\right\}=1.
\]
\item For the fixed partition and pilot residual scores, each $B\in\mathcal B$
and $\theta\in B$ satisfy, for almost every realization of $\mathcal D_0$,
\begin{align}
 (Q_{\theta,h}\otimes P_h^{\rm aux})
 \{p_B^\rho(G)\le u\mid\mathcal D_0\}&\le u,
 \qquad 0\le u\le1,
 \label{eq:main-residual-size}\\
 \shortintertext{\textit{and}}
 (Q_{\theta,h}\otimes P_h^{\rm aux})
 \{\gamma^2/2\in C_\alpha^\rho(G)\mid\mathcal D_0\}
     &\ge1-\alpha.
 \label{eq:main-residual-coverage}
\end{align}
\end{enumerate}
\end{theorem}

\subsection{Transferring calibration to continuous chaos}
\label{subsec:transfer-main}

Write $P_\theta$ for the continuous square-domain graph law; for
$\theta=(\gamma,n)$, also write $P_{\gamma,n}=P_\theta$ and
$Q_{\gamma,n,h}=Q_{\theta,h}$.  Let $K_h$ be a bounded continuous or
finite-pixel covariance approximation with chaos measure $\mu_{\gamma,h}$,
and assume, for each fixed $\gamma<2$,
\begin{equation}
 \mu_{\gamma,h}\Longrightarrow\mu_\gamma
 \quad\text{in distribution for the weak topology of finite measures on }\overline D.
\end{equation}
For the quantitative comparison, assume in addition that $K$ has bounded
continuous regularizations $K_\delta$, with chaos measures $\mu_{\gamma,\delta}$, such
that, for $x\ne y$,
\begin{align}
 K_\delta(x,y)&\longrightarrow K(x,y),\qquad
 K_\delta(x,y)\le\log_+\frac1{|x-y|}+B_1,\notag\\
 \shortintertext{\textit{and}}
 \mu_{\gamma,\delta}&\Longrightarrow\mu_\gamma
 \quad\text{in distribution for the weak topology}.
\end{align}
All comparisons use the same domain, intensity, and connection rule.

For $B=[\gamma_-,\gamma_+]\times[n_B^-,n_B^+]$ with
$\gamma_+<\sqrt2$, define
\begin{equation}
 \eta_B=1\wedge (n_B^+)^2
 \iint_{D^2}\sup_{\gamma\in[\gamma_-,\gamma_+]}
 \bigl|e^{\gamma^2K(x,y)}-e^{\gamma^2K_h(x,y)}\bigr|\dd x\dd y.
 \label{eq:eta-box}
\end{equation}
A certified upper bound may replace this integral.  With fixed simulator
resolution $h$, replace $p_B$ by
$\widetilde p_B=\min\{1,p_B+\eta_B\}$ and denote the corresponding retained
sets by $\widetilde\Theta_j$.  Total variation is normalized as
$d_{\TV}(P,Q)=\sup_A|P(A)-Q(A)|$.

\begin{theorem}
\label{thm:transfer}
For every fixed $n\ge1$ and $0\le\gamma<2$ under the preceding
weak-convergence assumption,
\begin{equation}
 d_{\TV}(P_{\gamma,n},Q_{\gamma,n,h})\longrightarrow0
 \quad\text{as }h\to\infty.
 \label{eq:main-qualitative}
\end{equation}
In the range $\gamma^2<2$, the discrepancy obeys the fixed-$h$ bound
\begin{equation}
 d_{\TV}(P_{\gamma,n},Q_{\gamma,n,h})
 \le 1\wedge n^2\iint_{D^2}
 \bigl|e^{\gamma^2K(x,y)}-e^{\gamma^2K_h(x,y)}\bigr|\dd x\dd y.
 \label{eq:main-tv}
\end{equation}
Consequently, if $\Gamma<\sqrt2$, the corrected refinement procedure satisfies
\begin{equation}
 \sup_{\theta\in\Theta}
 (P_\theta\otimes P_h^{\rm aux})
 \{\exists j\ge0:\theta\notin\widetilde\Theta_j\}\le\alpha.
 \label{eq:main-corrected}
\end{equation}
\end{theorem}

The condition $\Gamma<\sqrt2$ is the integrability boundary of the direct
second-order comparison.  For the periodic sequence below, count truncation
extends the quantitative approximation to every $\Gamma<2$.

We next define the periodic approximation used in Section~\ref{sec:numerics}.
Let $\T^2=\R^2/\Z^2$, write $d_{\T}$ for its flat geodesic distance, represent
each point by its unique element of $[0,1)^2$, let $h\ge8$ be even, and set
\begin{equation}
 \pi_hx=h^{-1}\lfloor hx\rfloor,\qquad
 \Lambda_h=(\Z/h\Z)^2,\qquad
 I_h=\{-h/2,\ldots,h/2-1\}\times
     \{-h/2+1,\ldots,h/2\}.
 \label{eq:fft-indexing}
\end{equation}
For $r\in\Lambda_h$, let $\widetilde r$ be its unique representative in $I_h$
and define
\begin{equation}
 \widehat K_h(r)=
 \begin{cases}
 \displaystyle\frac{e^{-9|\widetilde r|^2/h^2}}
                       {2\pi|\widetilde r|^2},&r\ne0,\\[2mm]
 0,&r=0,
 \end{cases}
 \qquad
 \phi_{h,r}(x)=e^{2\pi i\widetilde r\cdot\pi_hx}.
 \label{eq:fft-weights}
\end{equation}
The finite-group covariance and its pixel chaos are
\begin{align}
 K_h^{\T}(x,y)
 &=\sum_{r\in\Lambda_h}\widehat K_h(r)
       \phi_{h,r}(x)\overline{\phi_{h,r}(y)},
 &v_h&=K_h^{\T}(x,x)=\sum_{r\in\Lambda_h}\widehat K_h(r),
 \label{eq:fft-covariance}\\
 \shortintertext{and}
 \mu_{\gamma,h}^{\T}(\dd x)
 &=e^{\gamma X_h(\pi_hx)-\gamma^2v_h/2}\dd x,
 &&\Cov(X_h(\pi_hx),X_h(\pi_hy))=K_h^{\T}(x,y).
 \label{eq:fft-chaos}
\end{align}
The continuous target is the zero-mode torus field with covariance
\begin{equation}
 K_{\T}(z)=\sum_{k\in\Z^2\setminus\{0\}}
       \frac{e^{2\pi i k\cdot z}}{2\pi|k|^2},
 \label{eq:torus-target}
\end{equation}
where the series converges in $L^2(\T^2)$ and hence also distributionally; its
two-variable covariance kernel is $(x,y)\mapsto K_{\T}(x-y)$.  Let
$\mu_\gamma^{\T}$ be its full-variance subcritical GMC, which is almost surely
finite and non-atomic for $\gamma<2$; canonical torus GMC constructions appear
in \cite{DavidRhodesTorus}.  From $\mu_{\gamma,h}^{\T}$ and
$\mu_\gamma^{\T}$, respectively, sample vertices at intensity $n$ and connect
them independently with probability
\begin{equation}
 k_n^{\T}(x,y)=e^{-\pi n d_{\T}(x,y)^2}.
 \label{eq:torus-connection}
\end{equation}
After deleting positions and labels while retaining isolated vertices, denote
the two finite-graph laws by $Q_{\gamma,n,h}^{\T}$ and
$P_{\gamma,n}^{\T}$.  For $\theta=(\gamma,n)$, write
$Q_{\theta,h}^{\T}=Q_{\gamma,n,h}^{\T}$,
$P_\theta^{\T}=P_{\gamma,n}^{\T}$, and $\nu(\theta)=\gamma^2/2$.
Before connection, let $\Xi_{\gamma,n,h}^{\T}$ and
$\Xi_{\gamma,n}^{\T}$ denote the corresponding annealed Cox point
processes on $\T^2$.
For $\Theta$ in \eqref{eq:parameter-domain}, put
\begin{equation}
 \Delta_h(\Theta)=
 \sup_{(\gamma,n)\in\Theta}
 d_{\TV}(Q_{\gamma,n,h}^{\T},P_{\gamma,n}^{\T}).
 \label{eq:periodic-discrepancy}
\end{equation}

For $a\ge0$ and $0<s<1$, put
\begin{equation}
 \beta(a,s)=(1-s)\left(2-\frac a2(2-s)\right),\qquad
 \beta_{\Gamma}(s)=\beta(\Gamma^2,s).
 \label{eq:rate-exponents}
\end{equation}
For every $0<\Gamma<2$ there exist $0<s<1$ and
$1<p<4/\Gamma^2$ such that $0<\beta_{\Gamma}(s)<1$.  For $0\le\Gamma<\sqrt2$, also set
\begin{equation}
 \omega_\Gamma(h)=
 \begin{cases}
  h^{-1},&0\le\Gamma<1,\\
  h^{-1}\log h,&\Gamma=1,\\
  h^{-(2-\Gamma^2)},&1<\Gamma<\sqrt2.
 \end{cases}
 \label{eq:second-moment-rate}
\end{equation}

\begin{theorem}
\label{thm:periodic-rate}
Fix $0<\Gamma<2$, choose $0<s<1$ and $1<p<4/\Gamma^2$ with
$0<\beta_{\Gamma}(s)<1$, and let $h\ge8$ be even and $\bar n\ge1$.
There is $C=C(\Gamma,p,s)<\infty$, independent of $h,\bar n,\gamma$, and $n$,
such that
\begin{equation}
 \sup_{\substack{0\le\gamma\le\Gamma\\1\le n\le\bar n}}
 \max\left\{
 d_{\TV}\bigl(\Law(\Xi_{\gamma,n,h}^{\T}),\Law(\Xi_{\gamma,n}^{\T})\bigr),
 d_{\TV}(Q_{\gamma,n,h}^{\T},P_{\gamma,n}^{\T})
 \right\}
 \le 1\wedge C
 \left(\bar n^2h^{-\beta_{\Gamma}(s)}\right)^{p/(p+s)}.
 \label{eq:periodic-full-rate}
\end{equation}
If $0<\Gamma<\sqrt2$, there is also $C=C(\Gamma)<\infty$ such that
\begin{equation}
 \sup_{\substack{0\le\gamma\le\Gamma\\1\le n\le\bar n}}
 \max\left\{
 d_{\TV}\bigl(\Law(\Xi_{\gamma,n,h}^{\T}),\Law(\Xi_{\gamma,n}^{\T})\bigr),
 d_{\TV}(Q_{\gamma,n,h}^{\T},P_{\gamma,n}^{\T})
 \right\}
 \le1\wedge C\bar n^2\omega_\Gamma(h).
 \label{eq:periodic-second-moment-rate}
\end{equation}
Consequently, for every fixed $\Theta$ in \eqref{eq:parameter-domain},
\begin{equation}
 \Delta_h(\Theta)
 \le1\wedge C_{\Theta,p,s}
       h^{-\beta_{\Gamma}(s)p/(p+s)}
 \longrightarrow0
 \quad\text{as }h\to\infty\text{ through even integers}.
 \label{eq:main-periodic-tv}
\end{equation}
For $\Gamma=0$, both distances are identically zero.
\end{theorem}

For the range $\sqrt2\le\Gamma<2$, one convenient explicit choice is
\begin{equation}
 s_\Gamma=\frac{3\Gamma^2-4}{2\Gamma^2},\qquad
 p_\Gamma=\frac{\Gamma^2+4}{2\Gamma^2}.
 \label{eq:hard-range-choice}
\end{equation}
It gives
\begin{equation}
 \beta_\Gamma(s_\Gamma)=
 \frac{(4-\Gamma^2)^2}{8\Gamma^2},\qquad
 \kappa_\Gamma=\beta_\Gamma(s_\Gamma)
 \frac{p_\Gamma}{p_\Gamma+s_\Gamma}
 =\frac{(4-\Gamma^2)^2(\Gamma^2+4)}{32\Gamma^4}.
 \label{eq:hard-range-exponents}
\end{equation}
Thus the right side of \eqref{eq:periodic-full-rate} may be written as
$1\wedge C(\Gamma)(\bar n^2h^{-\beta_\Gamma(s_\Gamma)})^{p_\Gamma/(p_\Gamma+s_\Gamma)}$ in this range.

For each even $h\ge8$, run either procedure of Theorem~\ref{thm:calibration}
with a fresh calibration sample under $Q_{\cdot,h}^{\T}$.  According to the
procedure used, set
\begin{equation}
 F_{\theta,h}=
 \begin{cases}
  \{\exists j\ge0:\theta\notin\Theta_j\},&\text{fixed-statistic refinement},\\
  \{\nu(\theta)\notin C_\alpha^\rho(G)\},&\text{residual calibration}.
 \end{cases}
 \label{eq:periodic-failure-event}
\end{equation}
Let $C_h$ be the resulting $\nu$-projected confidence set, and let all
auxiliary randomness have law $P_h^{\rm aux}$.

\begin{corollary}
\label{cor:fft-continuum-coverage}
If $0<\Gamma<2$ and $p,s$ satisfy the conditions of
Theorem~\ref{thm:periodic-rate}, then
\begin{equation}
 \sup_{\theta\in\Theta}
 (P_\theta^{\T}\otimes P_h^{\rm aux})(F_{\theta,h})
 \le\alpha+\Delta_h(\Theta)
 \le\alpha+\bigl(1\wedge C_{\Theta,p,s}
                 h^{-\beta_{\Gamma}(s)p/(p+s)}\bigr),
 \label{eq:periodic-coverage-error}
\end{equation}
and
\begin{equation}
 \liminf_{\substack{h\to\infty\\ h\ {\rm even}}}
 \inf_{\theta\in\Theta}
 (P_\theta^{\T}\otimes P_h^{\rm aux})
 \{\nu(\theta)\in C_h\}\ge1-\alpha.
 \label{eq:periodic-asymptotic-coverage}
\end{equation}
For $\Gamma=0$, the same conclusions hold with $\Delta_h(\Theta)=0$.
\end{corollary}

\section{Preliminaries}
\label{sec:preliminaries}

\subsection{Gaussian multiplicative chaos}
\label{subsec:gmc-input}

We first record the GMC approximation theorem, using the notation of this
paper.  Let
$\mathcal H$ be a real separable Hilbert space carrying an isonormal Gaussian
process, let $(E,m)$ be a finite measure space, and let
$X_j,X:\mathcal H\to L^0(E,m)$ be continuous linear maps representing
generalized Gaussian fields.  Denote their covariance kernels by $K_j$ and $K$,
respectively, and let $\mu_j$ be the GMC over $X_j$ normalized by $\EE\mu_j=m$.

\begin{proposition}
\label{prop:gmc-approximation}
Suppose
\begin{equation}
 X_j\xi\longrightarrow X\xi
 \quad\text{in }L^0(E,m)\quad(\xi\in\mathcal H),
 \qquad
 K_j\longrightarrow K
 \quad\text{in }L^0(E\times E,m\otimes m),
 \label{eq:shamov-assumptions}
\end{equation}
and $(\mu_j(E))_{j\ge1}$ is uniformly integrable.  Then the GMC $\mu$ over $X$
exists and, for every $f\in L^1(E,m)$,
\begin{equation}
 \int_E f\,\dd\mu_j\longrightarrow\int_E f\,\dd\mu
 \quad\text{in }L^1.
 \label{eq:shamov-criterion}
\end{equation}
\end{proposition}
\smallskip
\noindent\emph{Source.} \cite[Theorem~25]{Shamov}.

The exact-scale logarithmic reference construction and the covariance
comparison used below are \cite[Theorems~27--28]{Shamov}.  In particular,
Theorem~28 applies to pointwise ordered covariance kernels; it does not require
their difference to be a covariance kernel.  We also use Kahane's convexity
inequality in the form stated in \cite[Theorem~2.1]{RhodesVargas}.

We also record the moment statements needed later, specialized to dimension
two.

\begin{proposition}
\label{prop:gmc-moments}
Let $0<\gamma<2$.  For every compact set $A\subset D$ and every $q$ with
$0<q<4/\gamma^2$,
\begin{equation}
 \EE\bigl[\mu_\gamma(A)^q\bigr]<\infty.
 \label{eq:gmc-positive-moments}
\end{equation}
For every nonempty compact Euclidean ball $B\subset D$ and every $q>0$,
\begin{equation}
 \EE\bigl[\mu_\gamma(B)^{-q}\bigr]<\infty.
 \label{eq:gmc-moments}
\end{equation}
\end{proposition}
\smallskip
\noindent\emph{Source.} \cite[Theorems~2.11--2.12]{RhodesVargas}.

For the multifractal statements, we use the notation of
\cite[Definition~2.9]{Bertacco}, specialized to $q=2$:
\begin{equation}
 \operatorname{Pack}_2(\mu,r)=\sup\left\{\sum_i\mu(B(x_i,r))^2:
          B(x_i,r)\text{ are pairwise disjoint},\ x_i\in D\right\}.
 \label{eq:packing-sum}
\end{equation}
For $0<\gamma<2$, set
\begin{equation}
 F_\alpha=\left\{x\in D:
       \lim_{r\downarrow0}\frac{\log\mu_\gamma(B(x,r))}{\log r}=\alpha\right\},
 \label{eq:local-dimension-set}
\end{equation}
and define
\begin{align}
 \tau_\gamma&=
  \begin{cases}
   2-\gamma^2,&0\le\gamma\le1,\\
   (2-\gamma)^2,&1<\gamma<2,
  \end{cases}
 \label{eq:correlation-exponent}\\
 \shortintertext{and}
 f_\gamma(\alpha)&=2-\frac12\left(
       \frac{2-\alpha}{\gamma}+\frac\gamma2\right)^2.
 \label{eq:singularity-spectrum}
\end{align}
The next statements are the corresponding dimension-two multifractal results.
Set
\begin{equation}
 I_\gamma=
 \left[\left(\sqrt2-\frac{\gamma}{\sqrt2}\right)^2,
       \left(\sqrt2+\frac{\gamma}{\sqrt2}\right)^2\right].
 \label{eq:multifractal-range}
\end{equation}

\begin{proposition}
\label{prop:multifractal}
For each fixed $0<\gamma<2$, almost surely,
\begin{equation}
 \limsup_{r\downarrow0}
 \frac{\log \operatorname{Pack}_2(\mu_\gamma,r)}{-\log r}=-\tau_\gamma,
 \label{eq:packing-exponent}
\end{equation}
and
\begin{equation}
 \dim_H F_\alpha=
 \begin{cases}
  f_\gamma(\alpha),&\alpha\in I_\gamma,\\
  0,&\alpha\notin I_\gamma.
 \end{cases}
 \label{eq:bertacco-spectrum}
\end{equation}
Moreover, for $\mu_\gamma$-almost every $x$,
\begin{equation}
 \lim_{r\downarrow0}\frac{\log\mu_\gamma(B(x,r))}{\log r}
 =2-\frac{\gamma^2}{2}=2-\nu.
 \label{eq:exact-dimension}
\end{equation}
\end{proposition}
\smallskip
\noindent\emph{Source.} \cite[Theorem~3.1 and Proposition~3.3]{Bertacco}.

\subsection{Poisson and Monte Carlo identities}
\label{subsec:poisson-ranks}

For a Poisson process $\Xi$ with finite intensity measure $\lambda$, Mecke's
identity is
\begin{equation}
 \EE\sum_{x\in\Xi}F(x,\Xi-\delta_x)
 =\int \EE F(x,\Xi)\,\lambda(\dd x),
 \qquad F\ge0.
 \label{eq:mecke}
\end{equation}
We use its multivariate form for edge-marked Poisson processes as in
\cite[Section~3]{LastNestmannSchulte}.  Poisson intensity perturbations and
add-one difference operators are treated in \cite{MolchanovZuyev,LastPerturbation}.

For Monte Carlo inversion, we use the exchangeable-rank construction with
ties and its conservative weak-tail form.

\begin{proposition}
\label{prop:rank}
If $T_0,T_1,\ldots,T_R$ are exchangeable real random variables, then
\begin{equation}
 q^+=\frac{1+\sum_{r=1}^R\one\{T_r\ge T_0\}}{R+1},\qquad
 q^-=\frac{1+\sum_{r=1}^R\one\{T_r\le T_0\}}{R+1}
 \label{eq:exact-ranks}
\end{equation}
satisfy
\begin{equation}
 \Pr(q^\pm\le u)\le u,
 \qquad 0\le u\le1.
 \label{eq:rank-validity}
\end{equation}
\end{proposition}
\smallskip
\noindent\emph{Source.} \cite[Section~2.2, especially (2.32)--(2.33) and
Proposition~2.4]{Dufour}; see also \cite[Section~3, (7)]{GlazerStark}.

\section{Single-graph statistics}
\label{sec:statistics}

\subsection{Conditional edge fluctuations}
\label{subsec:edges}

Conditioning on a deterministic environment isolates the graph randomness.  We
first express the conditional edge mean and variance through $C_r$ and
$C_r^{(2)}$.  Fix a nonzero finite non-atomic measure $\mu$ on $D$, write
$M=\mu(D)$, and let $\Xi_n$ be the Poisson vertex process with intensity
$n\mu$.  For $r>0$, write
\begin{equation}
 \begin{aligned}
  k_r(x,y)&=e^{-|x-y|^2/r^2},&
  a_r(x)&=\int_Dk_r(x,y)\mu(\dd y),\\
  C_r&=\int_Da_r(x)\mu(\dd x),&
  C_r^{(2)}&=\int_Da_r(x)^2\mu(\dd x).
 \end{aligned}
 \label{eq:conditional-integrals}
\end{equation}

\begin{lemma}
\label{lem:edge-fluctuations}
For $r_n=cn^{-1/2}$,
\begin{align}
 e_n(\mu):=\EE(E_n\mid\mu)&=\tfrac12n^2C_{r_n},
 \label{eq:edge-mean}\\
 \shortintertext{\textit{and}}
 \Var(E_n\mid\mu)&=\tfrac12n^2C_{r_n}+n^3C_{r_n}^{(2)}.
 \label{eq:edge-variance}
\end{align}
There is a constant $C_c<\infty$, independent of $\mu$ and $n$, such that
\begin{equation}
 \EE\left[\left(\frac{E_n}{e_n(\mu)}-1\right)^2\middle|\mu\right]
 \le C_c\left(\frac1{nM^2}+\frac1{\sqrt n\,M}\right).
 \label{eq:conditional-concentration}
\end{equation}
\end{lemma}

\begin{proof}
Conditional on $\Xi_n$, the edge indicators are independent.  The first- and
second-order Mecke formulas give, for $r=r_n$,
\begin{align*}
 \EE(E_n\mid\mu)
 &=\frac{n^2}{2}\iint_{D^2}k_r(x,y)\mu(\dd x)\mu(\dd y),\\
 \EE[\Var(E_n\mid\Xi_n,\mu)\mid\mu]
 &=\frac{n^2}{2}\iint_{D^2}k_r(x,y)(1-k_r(x,y))
      \mu(\dd x)\mu(\dd y),\\
 \Var(\EE[E_n\mid\Xi_n,\mu]\mid\mu)
 &=\frac{n^2}{2}\iint_{D^2}k_r(x,y)^2\mu(\dd x)\mu(\dd y)
   +n^3C_r^{(2)}.
\end{align*}
Hence the law of total variance yields \eqref{eq:edge-mean}--\eqref{eq:edge-variance}.

Because $k_r$ is positive definite and $k_r(x,x)=1$, Cauchy--Schwarz in
its reproducing kernel Hilbert space gives
\begin{equation}
 \begin{aligned}
 a_r(x)^2
 &\le \|k_r(x,\cdot)\|^2
 \left\|\int_Dk_r(y,\cdot)\mu(\dd y)\right\|^2=C_r,\\
 C_r^{(2)}
 &\le \bigl(\sup_{x\in D}a_r(x)\bigr)C_r\le C_r^{3/2}.
 \end{aligned}
 \label{eq:kernel-cs}
\end{equation}
Consequently,
\begin{equation}
 \frac{\Var(E_n\mid\mu)}{e_n(\mu)^2}
 \le\frac2{n^2C_{r_n}}+\frac4{n\sqrt{C_{r_n}}}.
 \label{eq:relative-variance}
\end{equation}

Partition $D$ into at most $C_Dr^{-2}$ squares of diameter at most $r$.
Since $k_r(x,y)\ge e^{-1}$ on each square,
\begin{equation}
 C_r\ge e^{-1}\sum_Q\mu(Q)^2
 \ge e^{-1}\frac{M^2}{\#\{Q\}}
 \ge c_DM^2r^2.
 \label{eq:correlation-lower}
\end{equation}
Substituting $r_n=cn^{-1/2}$ into \eqref{eq:relative-variance} and using
\eqref{eq:correlation-lower} proves \eqref{eq:conditional-concentration}.
\end{proof}

\subsection{The correlation integral in a typical environment}
\label{subsec:correlation}

Lemma~\ref{lem:edge-fluctuations} reduces the edge statistic to the small-scale
behavior of $C_r$.  We now specialize the environment to GMC, writing
\[
 \mu=\mu_\gamma,\qquad M=M_\gamma.
\]
The next lemma identifies the corresponding exponent.

\begin{lemma}
\label{lem:correlation-limit}
For each $0\le\gamma<2$, for almost every realization of $\mu_\gamma$,
\begin{equation}
 \lim_{r\downarrow0}\frac{\log C_r}{\log r}=\tau_\gamma.
 \label{eq:correlation-limit}
\end{equation}
\end{lemma}

\begin{proof}
For $\gamma=0$, direct integration gives $C_r\asymp r^2$.  Fix
$0<\gamma<2$.  We first bound $C_r$ from above.  Use
\eqref{eq:packing-exponent} and \eqref{eq:bertacco-spectrum} on the
probability-one event where they hold for all
\[
 \alpha(t)=2+\frac{\gamma^2}{2}-\gamma t,
 \qquad t\in\mathbb Q\cap(-2,2).
\]
For every $\epsilon>0$,
\begin{equation}
 \operatorname{Pack}_2(\mu,r)\le r^{\tau_\gamma-\epsilon}
 \quad\text{for all sufficiently small }r.
 \label{eq:packing-upper}
\end{equation}
Fix $0<\delta<1$ and put $r_\delta=r^{1-\delta}$.  Then
\begin{equation}
 C_r\le\int_D\mu(B(x,r_\delta))\mu(\dd x)+M^2e^{-r^{-2\delta}}.
 \label{eq:kernel-tail}
\end{equation}

Let $(D_i)_i$ be the square grid of side $r_\delta$, and write $j\sim i$ when
$\operatorname{dist}(D_i,D_j)\le r_\delta$.  The number of $j\sim i$ is bounded by an
absolute constant.  Hence
\begin{align}
 \int_D\mu(B(x,r_\delta))\mu(\dd x)
 &\le \sum_i\mu(D_i)\sum_{j\sim i}\mu(D_j) \\
 &\le \frac12\sum_i\sum_{j\sim i}
       \bigl(\mu(D_i)^2+\mu(D_j)^2\bigr)
 \le C_1\sum_i\mu(D_i)^2.
 \label{eq:grid-neighbor}
\end{align}
Color the grid by the $16$ residue classes of its two coordinates modulo $4$.
For every cell with positive mass choose $x_i\in D_i\cap D$.  Within one color
class the balls $B(x_i,\sqrt2r_\delta)$ are pairwise disjoint and
$D_i\subset B(x_i,\sqrt2r_\delta)$.  Hence
\begin{equation}
 \sum_i\mu(D_i)^2
 \le16\operatorname{Pack}_2(\mu,\sqrt2r_\delta),
 \qquad
 \int_D\mu(B(x,r_\delta))\mu(\dd x)
 \le C_2\operatorname{Pack}_2(\mu,\sqrt2r_\delta).
 \label{eq:grid-packing}
\end{equation}
Equations~\eqref{eq:packing-upper}, \eqref{eq:kernel-tail}, and
\eqref{eq:grid-packing} imply
\[
 \liminf_{r\downarrow0}\frac{\log C_r}{\log r}
 \ge(1-\delta)(\tau_\gamma-\epsilon).
\]
Let $\delta,\epsilon\downarrow0$.

For the matching lower bound on $C_r$, fix $t\in\mathbb Q\cap(-2,2)$, set
$\alpha=\alpha(t)$, and choose
\[
 0<s<f_\gamma(\alpha),\qquad \epsilon>0.
\]
For $m\ge1$, define
\[
 F_{\alpha,m}^{(\epsilon)}
 =\left\{x\in F_\alpha:
 \operatorname{dist}(x,\partial D)\ge m^{-1},\quad
 \mu(B(x,\rho))\ge\rho^{\alpha+\epsilon}
 \ \text{for all }0<\rho<m^{-1}\right\}.
\]
Since $F_{\alpha,m}^{(\epsilon)}\uparrow F_\alpha$, countable stability of
Hausdorff dimension gives $m$ such that, with
$F=F_{\alpha,m}^{(\epsilon)}$ and $r_0=m^{-1}$,
\begin{equation}
 \dim_HF>s,
 \qquad
 \mu(B(x,r))\ge r^{\alpha+\epsilon}
 \quad(x\in F,\ 0<r<r_0).
 \label{eq:uniform-level}
\end{equation}
Let $x_1,\ldots,x_{N_r}$ be a maximal $3r$-separated subset of $F$.
Then $B(x_j,3r)$ cover $F$, while $B(x_j,r)$ are pairwise disjoint.  Since
$\dim_HF>s$, for all sufficiently small $r$,
\begin{equation}
 1\le\mathcal H^s_{6r}(F)
 \le N_r(6r)^s,
 \qquad
 N_r\ge6^{-s}r^{-s}.
 \label{eq:hausdorff-count}
\end{equation}
Moreover $k_r(x,y)\ge e^{-4}$ on $B(x_j,r)^2$, so
\begin{equation}
 C_r\ge e^{-4}\sum_{j=1}^{N_r}\mu(B(x_j,r))^2
 \ge c r^{-s+2\alpha+2\epsilon}.
 \label{eq:level-lower}
\end{equation}
Letting $s\uparrow f_\gamma(\alpha)$ and $\epsilon\downarrow0$ gives
\[
 \limsup_{r\downarrow0}\frac{\log C_r}{\log r}
 \le\inf_{t\in\mathbb Q\cap(-2,2)}
       \{2\alpha(t)-f_\gamma(\alpha(t))\}.
\]
By continuity the infimum is unchanged over $-2<t<2$, and
\[
 2\alpha(t)-f_\gamma(\alpha(t))
 =2+\gamma^2-2\gamma t+\frac{t^2}{2}.
\]
Therefore
\[
 \inf_{-2<t<2}\left(2+\gamma^2-2\gamma t+\frac{t^2}{2}\right)
 =\begin{cases}
 2-\gamma^2,&0<\gamma\le1,\\
 (2-\gamma)^2,&1<\gamma<2,
 \end{cases}
 =\tau_\gamma.
\]
For $\gamma\ge1$ the boundary value is understood as $t\uparrow2$.
This proves \eqref{eq:correlation-limit}.
\end{proof}

\subsection{Mass-typical degrees}
\label{subsec:typical-degree}

Continue with $\mu=\mu_\gamma$.  For the degree statistic, the relevant quantity is the local connection mass
$a_{r_n}(x)$.  Exact dimensionality determines its scale at
$\mu_\gamma$-typical points, and the Palm description transfers that scale to
vertex degrees.  For $t>0$, set
\begin{equation}
 Z_n(t)=\#\left\{v:
   \left|\log(1+d_v)-\frac\nu2\log n\right|>t\right\}.
 \label{eq:bad-vertex-count}
\end{equation}
The next lemma shows that, at thresholds proportional to $\log n$, these
atypical vertices form a vanishing fraction of the graph.

\begin{lemma}
\label{lem:typical-degree}
For each $0\le\gamma<2$, for almost every realization of $\mu_\gamma$ and every $\epsilon>0$,
\begin{equation}
 \frac{Z_n(\epsilon\log n)}{N_n}\one_{\{N_n>0\}}
 \longrightarrow0
 \quad\text{in conditional probability}.
 \label{eq:bad-fraction-limit}
\end{equation}
\end{lemma}

\begin{proof}
By \eqref{eq:exact-dimension}, for
$\mu_\gamma$-almost every $x$ and every $\eta>0$,
\begin{equation}
 r^{2-\nu+\eta}\le\mu(B(x,r))\le r^{2-\nu-\eta}
 \quad\text{for all sufficiently small }r.
 \label{eq:local-mass-window}
\end{equation}
For $0<\delta<1$,
\begin{equation}
 e^{-1}\mu(B(x,r))\le a_r(x)
 \le\mu(B(x,r^{1-\delta}))+Me^{-r^{-2\delta}}.
 \label{eq:local-kernel}
\end{equation}
Equations~\eqref{eq:local-mass-window}--\eqref{eq:local-kernel} give
\[
 (1-\delta)(2-\nu-\eta)
 \le\liminf_{r\downarrow0}\frac{\log a_r(x)}{\log r}
 \le\limsup_{r\downarrow0}\frac{\log a_r(x)}{\log r}
 \le2-\nu+\eta.
\]
Letting $\delta,\eta\downarrow0$ and then taking $r_n=cn^{-1/2}$ yields
\begin{equation}
 \frac{\log a_r(x)}{\log r}\longrightarrow2-\nu,
 \qquad
 \frac{\log(na_{r_n}(x))}{\log n}\longrightarrow\frac\nu2.
 \label{eq:local-degree-mean}
\end{equation}
Thus the Palm mean degree at a mass-typical location has exponent $\nu/2$.
Under the Palm law with a vertex inserted at $x$, independent edge thinning
gives
\begin{equation}
 d_x\mid\mu\sim\operatorname{Poisson}(na_{r_n}(x)).
 \label{eq:palm-degree}
\end{equation}
Define
\[
 p_n(x,\epsilon)
 =\Pr\left(
 \left|\frac{\log(1+d_x)}{\log n}-\frac\nu2\right|>\epsilon
 \middle|\mu\right)
\]
under this Palm law.  If $\nu>0$, then
$\lambda_n=na_{r_n}(x)\to\infty$, and for $0<\eta<1$,
\[
 \Pr\left(\left|\frac{d_x}{\lambda_n}-1\right|>\eta\middle|\mu\right)
 \le\frac1{\eta^2\lambda_n}\longrightarrow0.
\]
Together with \eqref{eq:local-degree-mean}, this gives
$p_n(x,\epsilon)\to0$.  If $\gamma=0$, then
$na_{r_n}(x)\le\pi c^2$ and
\[
 p_n(x,\epsilon)
 \le\Pr\{\operatorname{Poisson}(\pi c^2)>n^\epsilon-1\}
 \longrightarrow0.
\]

Mecke's formula and dominated convergence now average the Palm statement over
all vertices:
\begin{equation}
 \frac1n\EE[Z_n(\epsilon\log n)\mid\mu]
 =\int_Dp_n(x,\epsilon)\mu(\dd x)\longrightarrow0.
 \label{eq:bad-count-conditional-mean}
\end{equation}
Thus $Z_n(\epsilon\log n)/n\to0$ in conditional probability, and, since
$N_n/n\to M>0$,
\begin{equation}
 \frac{Z_n(\epsilon\log n)}{N_n}\one_{\{N_n>0\}}
 =\frac{Z_n(\epsilon\log n)/n}{N_n/n}\one_{\{N_n>0\}}
 \longrightarrow0
 \quad\text{in conditional probability}.
 \label{eq:bad-fraction-ratio}
\end{equation}
This is \eqref{eq:bad-fraction-limit}.
\end{proof}

\subsection{Fractional moments of the local connection mass}
\label{subsec:fractional}

The almost-sure argument above is enough for consistency.  The uniform tail and
risk bounds require finite-scale estimates that hold uniformly over a compact
parameter range.  Fix $0<\bar\nu<2$, put $\bar\gamma=\sqrt{2\bar\nu}$, and set
\begin{equation}
 p_0=\min\left\{\frac12,\frac{2-\bar\nu}{2\bar\nu}\right\}.
 \label{eq:p0}
\end{equation}
Then $1+p_0<2/\bar\nu$.  Define
\begin{equation}
 J_+=\max\{1,\EE M_{\bar\gamma}^{1+p_0}\},\qquad
 J_-=\max\{1,\EE M_{\bar\gamma}^{-1}\}.
 \label{eq:total-mass-constants}
\end{equation}
The first quantity is finite by \cite[Theorem~2.11]{RhodesVargas}.  For a
nonempty compact ball $B_*\subset D$,
\begin{equation}
 M_{\bar\gamma}^{-1}\le\mu_{\bar\gamma}(B_*)^{-1},
 \qquad
 \EE\mu_{\bar\gamma}(B_*)^{-1}<\infty.
 \label{eq:negative-total-mass}
\end{equation}
by Proposition~\ref{prop:gmc-moments}, so $J_-<\infty$ as well.  Write
\begin{equation}
 \zeta_\nu(q)=(2+\nu)q-\nu q^2,
 \qquad
 J_0=J_+e^{2B_0\bar\nu(1+p_0)p_0}.
 \label{eq:moment-constants}
\end{equation}

\begin{lemma}
\label{lem:uniform-mass-moments}
For every square $Q\subset D$ of side $\ell\le1$, every
$0\le\nu\le\bar\nu$, and every $q\in[1-p_0,1+p_0]$,
\begin{equation}
 \EE\mu_\gamma(Q)^q\le J_0\ell^{\zeta_\nu(q)},
 \qquad \EE M_\gamma^{-1}\le J_-.
 \label{eq:uniform-mass-moments}
\end{equation}
\end{lemma}

\begin{proof}
Since $K\ge0$, covariance ordering and Kahane's convexity inequality give
\begin{equation}
 \EE M_\gamma^q\le\EE M_{\bar\gamma}^q\le J_+
 \quad(1\le q\le1+p_0),
 \qquad
 \EE M_\gamma^{-1}\le\EE M_{\bar\gamma}^{-1}\le J_-.
 \label{eq:total-mass-comparison}
\end{equation}
For $0<q<1$, Jensen's inequality gives $\EE M_\gamma^q\le1$.  We now pass
from the total mass to a square.  Let $Q=x_0+\ell D$ be axis-parallel, and put
$t=\log(1/\ell)$.  For $u,v\in D$,
\begin{equation}
 \left|K(x_0+\ell u,x_0+\ell v)-K(u,v)-t\right|\le2B_0.
 \label{eq:rescaled-covariance}
\end{equation}
For $1\le q\le1+p_0$, convex comparison with the reference field plus an
independent constant Gaussian of variance $t+2B_0$ yields
\begin{equation}
 \EE\mu_\gamma(Q)^q
 \le \ell^{2q}e^{\nu q(q-1)(t+2B_0)}\EE M_\gamma^q
 \le J_0\ell^{\zeta_\nu(q)}.
 \label{eq:local-positive-comparison}
\end{equation}
For $1-p_0\le q<1$, add variance $2B_0$ to the rescaled field and variance
$t$ to the reference field.  Concave comparison gives
\begin{equation}
 \EE\mu_\gamma(Q)^q
 \le e^{2B_0\nu q(1-q)}\ell^{\zeta_\nu(q)}\EE M_\gamma^q
 \le J_0\ell^{\zeta_\nu(q)}.
 \label{eq:local-concave-comparison}
\end{equation}
The comparisons are performed at a common regularization and then passed to
the chaos limit.  Equations~\eqref{eq:total-mass-comparison}--\eqref{eq:local-concave-comparison}
prove \eqref{eq:uniform-mass-moments}.
\end{proof}

We use a lattice kernel to transfer these mass moments to $a_r$.  Set
\begin{equation}
 w_z=\exp\left\{-\frac{((\|z\|_\infty-1)_+)^2}{8}\right\},\qquad
 W_0=\sum_{z\in\Z^2}w_z<\infty,
 \label{eq:lattice-weight}
\end{equation}
and let
\begin{equation}
 C_+=J_0W_0^{p_0},\qquad
 C_-=J_0e^{p_0}(2\sqrt2)^{2p_0+\bar\nu p_0^2}.
 \label{eq:kernel-moment-constants}
\end{equation}

\begin{lemma}
\label{lem:kernel-moments}
For $0<p\le p_0$, $0<r\le1$, and $0\le\nu\le\bar\nu$,
\begin{align}
 \EE\int_D a_r(x)^p\mu_\gamma(\dd x)
 &\le C_+r^{(2-\nu)p-\nu p^2},
 \label{eq:kernel-positive}\\
 \shortintertext{\textit{and}}
 \EE\int_D a_r(x)^{-p}\mu_\gamma(\dd x)
 &\le C_-r^{-(2-\nu)p-\nu p^2}.
 \label{eq:kernel-negative}
\end{align}
\end{lemma}

\begin{proof}
Choose a dyadic length
\[
 \frac{r}{2\sqrt2}<\ell\le\frac r{\sqrt2}
\]
and partition $D$ into dyadic squares $(D_i)_i$ of side $\ell$.  Put
$b_i=\mu_\gamma(D_i)$ and extend $(b_i)$ by zero to $\mathbb Z^2$.  Since
$\EE\mu_\gamma(\partial D_i)=0$, all grid boundaries have zero chaos mass
for almost every realization of $\mu_\gamma$.  For $x\in D_i$,
\begin{equation}
 a_r(x)\ge e^{-1}b_i,
 \qquad
 a_r(x)\le(w*b)_i.
 \label{eq:cell-kernel-bounds}
\end{equation}
Hence
\begin{equation}
 \int_D a_r(x)^{-p}\mu_\gamma(\dd x)
 \le e^p\sum_i b_i^{1-p}.
 \label{eq:negative-cell}
\end{equation}
H\"older's inequality and discrete Young's inequality give
\begin{equation}
 \int_D a_r(x)^p\mu_\gamma(\dd x)
 \le\sum_i b_i(w*b)_i^p
 \le\|b\|_{1+p}\|w*b\|_{1+p}^{p}
 \le W_0^p\sum_i b_i^{1+p}.
 \label{eq:positive-cell}
\end{equation}
Lemma~\ref{lem:uniform-mass-moments} yields
\begin{equation}
 \EE\sum_i b_i^{1\pm p}
 \le J_0\ell^{\zeta_\nu(1\pm p)-2},
 \label{eq:cell-moment-sums}
\end{equation}
where
\begin{equation}
 \zeta_\nu(1+p)-2=(2-\nu)p-\nu p^2\ge0,
 \qquad
 \zeta_\nu(1-p)-2=-(2-\nu)p-\nu p^2.
 \label{eq:cell-exponents}
\end{equation}
Using $\ell\le r$ in the positive estimate and
$\ell>r/(2\sqrt2)$ in the negative estimate gives
\eqref{eq:kernel-positive}--\eqref{eq:kernel-negative} with the constants in
\eqref{eq:kernel-moment-constants}.
\end{proof}

\subsection{A finite-scale count of atypical degrees}
\label{subsec:bad-degrees}

We now convert the fractional-moment bounds into a finite-$n$ estimate for the
count $Z_n(t)$ defined in \eqref{eq:bad-vertex-count}.  Set
\begin{equation}
 C_Z=1+C_++C_-c^{-(2p_0+\bar\nu p_0^2)}.
 \label{eq:bad-count-constant}
\end{equation}
\begin{lemma}
\label{lem:bad-count}
For $n\ge1$, $t>0$, $0<p\le p_0$, and $0\le\nu\le\bar\nu$,
\begin{equation}
 \EE Z_n(t)\le nC_Z\exp\left\{-pt+\frac{\bar\nu p^2}{2}\log n\right\}.
 \label{eq:bad-count-bound}
\end{equation}
\end{lemma}

\begin{proof}
For $D_\lambda\sim\operatorname{Poisson}(\lambda)$ and $0<p\le1$,
\begin{align}
 \EE(1+D_\lambda)^p
 &\le(1+\lambda)^p\le1+\lambda^p,
 \label{eq:poisson-positive}\\
 \shortintertext{\textit{and}}
 \EE(1+D_\lambda)^{-p}
 &\le\left(\EE(1+D_\lambda)^{-1}\right)^p
 =\left(\frac{1-e^{-\lambda}}\lambda\right)^p
 \le\lambda^{-p}.
 \label{eq:poisson-negative}
\end{align}
Mecke's identity, \eqref{eq:palm-degree}, and
Lemma~\ref{lem:kernel-moments} imply
\begin{align}
 \frac1n\EE\sum_v\left(\frac{1+d_v}{n^{\nu/2}}\right)^p
 &\le n^{-\nu p/2}
    +n^{(1-\nu/2)p}\EE\int_D a_{r_n}^p\,\dd\mu_\gamma
 \le(1+C_+)n^{\nu p^2/2},
 \label{eq:degree-positive-moment}\\
 \shortintertext{\textit{and}}
 \frac1n\EE\sum_v\left(\frac{n^{\nu/2}}{1+d_v}\right)^p
 &\le n^{(\nu/2-1)p}\EE\int_D a_{r_n}^{-p}\,\dd\mu_\gamma
 \le C_-c^{-(2p_0+\bar\nu p_0^2)}n^{\nu p^2/2}.
 \label{eq:degree-negative-moment}
\end{align}
The exponents of $n$ satisfy
\begin{equation}
 \left(1-\frac\nu2\right)p
 -\frac12\bigl((2-\nu)p-\nu p^2\bigr)
 =\left(\frac\nu2-1\right)p
 -\frac12\bigl(-(2-\nu)p-\nu p^2\bigr)
 =\frac{\nu p^2}{2}.
 \label{eq:degree-exponent-cancellation}
\end{equation}
Exponential Markov bounds in the two deviation directions therefore give
\[
 \EE Z_n(t)
 \le ne^{-pt}\left[(1+C_+)
 +C_-c^{-(2p_0+\bar\nu p_0^2)}\right]n^{\bar\nu p^2/2},
\]
which is \eqref{eq:bad-count-bound}.
\end{proof}

To pass from the atypical-vertex count to the median, set
\[
 m_n=\log(1+\med_v d_v),
\]
and control the random graph size $N_n$.

\begin{lemma}
\label{lem:size-control}
Uniformly for $0\le\nu\le\bar\nu$,
\begin{equation}
 \EE(N_n/n)=1,\qquad \EE\frac n{N_n+1}\le J_-.
 \label{eq:size-moments}
\end{equation}
For $A,U>0$,
\begin{align}
 \PP(N_n>Un)&\le U^{-1},&
 \PP(N_n\ge2,\ N_n<n/A)&\le2J_-/A,&
 \PP(N_n\le1)&\le2J_-/n.
 \label{eq:size-tails}
\end{align}
\end{lemma}

\begin{proof}
The first identity follows from $\EE M_\gamma=1$.  The Poisson identity gives
\[
 \EE\left[\frac n{N_n+1}\middle|M_\gamma\right]
 =\frac{1-e^{-nM_\gamma}}{M_\gamma}\le M_\gamma^{-1}.
\]
Moreover,
\[
 \{N_n\ge2,\ N_n<n/A\}
 \subseteq\left\{\frac n{N_n+1}>\frac A2\right\},
 \qquad
 \{N_n\le1\}
 \subseteq\left\{\frac n{N_n+1}\ge\frac n2\right\}.
\]
Markov's inequality and \eqref{eq:size-moments} give the three bounds in
\eqref{eq:size-tails}.
\end{proof}

\subsection{An analytic confidence interval}
\label{subsec:analytic-interval}

Lemmas~\ref{lem:bad-count} and \ref{lem:size-control} allow the latent scale
$n$ to be replaced by the observed size $N_n$.  This yields a confidence
interval based only on the graph.  For $0<\delta<1$, define
\begin{equation}
 A_\delta=6J_-/\delta,\quad U_\delta=3/\delta,\quad
 b_\delta=\max(\log A_\delta,\log U_\delta),\quad
 x_\delta=\log(6A_\delta C_Z/\delta).
 \label{eq:analytic-constants}
\end{equation}
For $L\ge0$ and $x>0$, set
\begin{equation}
 \psi(L,x)=\min_{0<p\le p_0}\left\{\frac{\bar\nu pL}{2}+\frac xp\right\}.
 \label{eq:chernoff-function}
\end{equation}
The minimizer is
\[
 p_*(L,x)=
 \begin{cases}
  \min\{p_0,\sqrt{2x/(\bar\nu L)}\},&L>0,\\
  p_0,&L=0,
 \end{cases}
\]
and $\psi(L,x)$ is nondecreasing in $L$.
For $N\ge2$, put
\begin{align}
 R_\delta(N)
 &=\frac{2\psi(\log N+\log A_\delta,x_\delta)+\bar\nu b_\delta}{\log N},
 \label{eq:analytic-radius}\\
 I_\delta(G_n)
 &=[\widehat\nu_n-R_\delta(N_n),\widehat\nu_n+R_\delta(N_n)]
       \cap[0,\bar\nu].
 \label{eq:analytic-interval}
\end{align}
For $N_n\le1$, take $I_\delta(G_n)=[0,\bar\nu]$.

\begin{corollary}
\label{cor:analytic-coverage}
For all $n\ge1$,
\begin{equation}
 \inf_{0\le\nu\le\bar\nu}
 \PP_{\nu,n}\{\nu\in I_\delta(G_n)\}\ge1-\delta.
 \label{eq:analytic-coverage}
\end{equation}
For fixed $\delta$,
\[
 R_\delta(N)=O((\log N)^{-1/2}),\qquad N\to\infty.
\]
\end{corollary}

\begin{proof}
On $N_n\ge2$, the failure probability of
$n/A_\delta\le N_n\le U_\delta n$ is at most $2\delta/3$ by
Lemma~\ref{lem:size-control}.  Put $t_n=\psi(\log n,x_\delta)$ and
$p_n^*=p_*(\log n,x_\delta)$.  Lemma~\ref{lem:bad-count} with $p=p_n^*$ gives
$\EE Z_n(t_n)\le nC_Ze^{-x_\delta}$.
If $|m_n-(\nu/2)\log n|>t_n$, then at least $N_n/2$ vertices are
counted by $Z_n(t_n)$.  Hence
\[
 \PP\{|m_n-(\nu/2)\log n|>t_n,\ N_n\ge n/A_\delta\}
 \le2A_\delta C_Ze^{-x_\delta}=\delta/3.
\]
On the complementary event,
\begin{align*}
 |T_D(G_n)-\nu|
 &\le\frac{2t_n+\nu|\log N_n-\log n|}{\log N_n}\\
 &\le\frac{2\psi(\log N_n+\log A_\delta,x_\delta)+\bar\nu b_\delta}{\log N_n}.
\end{align*}
Projection preserves this bound.  Eventually the minimizer of
\eqref{eq:chernoff-function} is interior, with
$\psi(L,x)=\sqrt{2\bar\nu Lx}$, which proves the radius assertion.
\end{proof}

\subsection{Final proof of Theorem~\ref{thm:single}}
\label{subsec:final-single}

Fix a GMC environment $\mu=\mu_\gamma$ for which
Lemmas~\ref{lem:correlation-limit} and \ref{lem:typical-degree} hold.
Lemma~\ref{lem:edge-fluctuations} gives
\[
 \frac{E_n}{e_n(\mu)}\longrightarrow1
 \quad\text{in conditional probability},
\]
while Lemma~\ref{lem:correlation-limit} gives
\begin{equation}
 \frac{\log e_n(\mu)}{\log n}
 =2+\frac{\log C_{r_n}}{\log n}+o(1)
 \longrightarrow2-\frac{\tau_\gamma}{2}.
 \label{eq:edge-log-limit}
\end{equation}
Since $2-\tau_\gamma/2>0$, $e_n(\mu)\to\infty$, and therefore
\[
 \frac{\log(E_n\vee1)}{\log n}
 \longrightarrow2-\frac{\tau_\gamma}{2}
 \quad\text{in conditional probability}.
\]
Also, by \eqref{eq:random-size},
\[
 \frac{N_n}{n}\longrightarrow M>0,
 \qquad
 \frac{\log(N_n\vee2)}{\log n}\longrightarrow1
 \quad\text{in conditional probability}.
\]
Hence
\[
 T_E^0(G_n)\longrightarrow1-\frac{\tau_\gamma}{2}=b(\nu).
\]

For every $\epsilon>0$, on $\{N_n\ge1\}$,
\begin{equation}
 \left\{\left|\frac{m_n}{\log n}-\frac\nu2\right|>\epsilon\right\}
 \subseteq\{Z_n(\epsilon\log n)\ge N_n/2\}.
 \label{eq:median-from-bad-fraction}
\end{equation}
Lemma~\ref{lem:typical-degree}, \eqref{eq:random-size}, and
\eqref{eq:median-from-bad-fraction} give
\[
 \frac{m_n}{\log n}\longrightarrow\frac\nu2
 \quad\text{in conditional probability}.
\]
Together with $\log(N_n\vee2)/\log n\to1$, this yields
\[
 T_D(G_n)=\frac{2m_n}{\log(N_n\vee2)}\longrightarrow\nu
 \quad\text{in conditional probability}.
\]
This proves \eqref{eq:main-limits}.

For the uniform bound, set
\[
 \epsilon_0=\min\{1,\bar\nu,4\bar\nu p_0\},\qquad
 w=\frac{\epsilon^2\log n}{64\bar\nu},\qquad
 t=\frac{\epsilon\log n}{4},\qquad
 p=\frac\epsilon{4\bar\nu}.
\]
On
$F=\{N_n\ge2,\ ne^{-w}\le N_n\le ne^w\}$, the event
$|m_n-(\nu/2)\log n|\le t$ implies
\begin{equation}
 |T_D(G_n)-\nu|
 \le\frac{\epsilon/2+\epsilon^2/64}
             {1-\epsilon^2/(64\bar\nu)}<\epsilon.
 \label{eq:median-conversion}
\end{equation}
The definition of $\epsilon_0$ ensures $p\le p_0$, a positive denominator,
and the final strict inequality.  Lemma~\ref{lem:size-control} gives
\[
 \PP(F^c)\le(1+2J_-)e^{-w}+\frac{2J_-}{n}.
\]
On $F$, a median deviation larger than $t$ implies
$Z_n(t)\ge ne^{-w}/2$.  Lemma~\ref{lem:bad-count} therefore gives
\begin{align*}
 \PP\{|m_n-(\nu/2)\log n|>t,F\}
 &\le2e^w C_Z
   \exp\{-pt+\bar\nu p^2\log n/2\}\\
 &=2C_Ze^{-w}.
\end{align*}
Since interval projection cannot increase the distance to
$\nu\in[0,\bar\nu]$,
\begin{equation}
 \{|\widehat\nu_n-\nu|>\epsilon\}
 \subseteq F^c\cup
 \{|m_n-(\nu/2)\log n|>t,\ F\}.
 \label{eq:tail-event-inclusion}
\end{equation}
Thus \eqref{eq:main-tail} holds with
\[
 A_1=1+2J_-+2C_Z,
 \qquad
 A_2=2J_-.
\]

Put $\Delta_n=|\widehat\nu_n-\nu|$ and
$b_0=\epsilon_0^2/(64\bar\nu)$.  Since $0\le\Delta_n\le\bar\nu$ and
\eqref{eq:main-tail} is available for $0<\epsilon\le\epsilon_0$, while
$\PP(\Delta_n>\epsilon)\le\PP(\Delta_n>\epsilon_0)$ for
$\epsilon\ge\epsilon_0$,
\begin{align}
 \EE\Delta_n^2
 &=2\int_0^{\epsilon_0}\epsilon\,\PP(\Delta_n>\epsilon)\,\dd\epsilon
   +2\int_{\epsilon_0}^{\bar\nu}\epsilon\,
       \PP(\Delta_n>\epsilon)\,\dd\epsilon \notag\\
 &\le \frac{64\bar\nu A_1}{\log n}
   +\bar\nu^2\left(A_1n^{-b_0}+\frac{A_2}{n}\right).
 \label{eq:risk-tail-integration}
\end{align}
Define
\begin{equation}
 A_3=64\bar\nu A_1
 +\bar\nu^2A_1\sup_{x\ge1}xe^{-b_0x}
 +\frac{\bar\nu^2A_2}{e}.
 \label{eq:risk-constant}
\end{equation}
For $n\ge e$, \eqref{eq:risk-tail-integration} gives
$\EE\Delta_n^2\le A_3/\log n$, proving \eqref{eq:main-risk}.

\section{Constructive calibration and parameter refinement}
\label{sec:calibration}

This section proves Theorem~\ref{thm:calibration} for the finite model
\eqref{eq:pixel-density} on a bounded domain with a fixed Euclidean or periodic distance.  Write
\[
 k_n(x,y)=\exp\{-n\operatorname{dist}(x,y)^2/c^2\},
\]
where $\operatorname{dist}$ denotes that fixed distance.
For a fixed statistic, the argument couples all parameters in a rectangle,
bounds the statistic on that rectangle, and compares the resulting box ranks
with the exact Monte Carlo rank at the true parameter.  The same envelope
construction is then used for the pilot-fitted residual score.

\subsection{Common nodes and edges over a parameter rectangle}
\label{subsec:common-graph}

The first step is a pathwise coupling of all graph laws $Q_{\theta,h}$ in one
parameter rectangle.  Fix a realization of the Gaussian vector $Y$ and let
\[
 B=[\gamma_-,\gamma_+]\times[n_B^-,n_B^+]\subset\Theta.
\]
For $x\in A_i$, define
\begin{equation}
 g_i(\gamma)=\gamma Y_i-\tfrac12\gamma^2v_i.
 \label{eq:log-intensity}
\end{equation}
Its minimum and maximum on $[\gamma_-,\gamma_+]$ are
\begin{align}
 g_{B,i}^-
 &=\min\{g_i(\gamma_-),g_i(\gamma_+)\},
 \label{eq:log-intensity-lower}\\
 \shortintertext{and}
 g_{B,i}^+
 &=
 \begin{cases}
 g_i\!\left(\Pi_{[\gamma_-,\gamma_+]}
              \left(Y_i/v_i\right)\right),&v_i>0,\\
 0,&v_i=0.
 \end{cases}
 \label{eq:log-intensity-upper}
\end{align}
Indeed, $g_i$ is concave, and $v_i=0$ implies $Y_i=0$ almost surely.  Put
\begin{equation}
 \lambda_B^\pm(x)=n_B^\pm e^{g_{B,i}^\pm},\qquad x\in A_i.
 \label{eq:intensity-envelope}
\end{equation}
Then, for every $(\gamma,n)\in B$,
\begin{equation}
 \lambda_B^-(x)\le n w_{\gamma,h}(x)\le\lambda_B^+(x).
 \label{eq:intensity-domination}
\end{equation}

Let
\begin{equation}
 \overline\lambda(x)=\sup_{\theta\in\Theta}n w_{\gamma,h}(x).
 \label{eq:global-mother-intensity}
\end{equation}
Because the partition is finite and $\Theta$ is compact,
$\int_D\overline\lambda(x)\,\dd x<\infty$ for every realized $Y$.
Let $\mathcal P$ be a unit-rate Poisson process on $D\times[0,\infty)$,
independent of $Y$, and define the global mother set
\begin{equation}
 \mathcal V=\{(x,s)\in\mathcal P:s\le\overline\lambda(x)\}
 =\{(x_i,s_i):i\in I\}.
 \label{eq:global-mother-set}
\end{equation}
Thus $|I|<\infty$ almost surely.  Assign $U_{ij}\sim\mathrm{Unif}(0,1)$ independently over the unordered pairs
$\{i,j\}\subset I$ and independently of $(Y,\mathcal P)$.
For $\theta=(\gamma,n)$ set
\begin{align}
 V_\theta
 &=\{i\in I:s_i\le n w_{\gamma,h}(x_i)\},
 \label{eq:parameter-node-set}\\
 E_\theta
 &=\bigl\{\{i,j\}\subset V_\theta:
          U_{ij}\le k_n(x_i,x_j)\bigr\},
 \qquad G_\theta=(V_\theta,E_\theta).
 \label{eq:parameter-edge-set}
\end{align}
For the rectangle $B$, define
\begin{align}
 V_B^\pm
 &=\{i\in I:s_i\le\lambda_B^\pm(x_i)\},
 \label{eq:node-envelopes}\\
 E_B^-
 &=\bigl\{\{i,j\}\subset V_B^-:
          U_{ij}\le k_{n_B^+}(x_i,x_j)\bigr\},
 \label{eq:edge-envelope-lower}\\
 \shortintertext{and}
 E_B^+
 &=\bigl\{\{i,j\}\subset V_B^+:
          U_{ij}\le k_{n_B^-}(x_i,x_j)\bigr\},
 \qquad G_B^\pm=(V_B^\pm,E_B^\pm).
 \label{eq:edge-envelope-upper}
\end{align}

\begin{lemma}
\label{lem:graph-envelope}
Under the common construction above, simultaneously for every $\theta\in B$,
\begin{equation}
 V_B^-\subseteq V_\theta\subseteq V_B^+,
 \qquad
 G_B^-\subseteq G_\theta\subseteq G_B^+.
 \label{eq:graph-envelope}
\end{equation}
Moreover, $G_\theta$ has law $Q_{\theta,h}$.
\end{lemma}

\begin{proof}
Poisson thinning in \eqref{eq:parameter-node-set} gives intensity
$n w_{\gamma,h}(x)\,\dd x$, and the independent pair marks give conditional
edge probabilities $k_n$.  Hence $G_\theta\sim Q_{\theta,h}$.  For
$\theta\in B$, \eqref{eq:intensity-domination} and the monotonicity of $k_n$
in $n$ give
\[
 \lambda_B^-(x)\le n w_{\gamma,h}(x)\le\lambda_B^+(x),
 \qquad
 k_{n_B^+}(x,y)\le k_n(x,y)\le k_{n_B^-}(x,y).
\]
Applying these inequalities to the same Poisson heights and pair marks proves
\eqref{eq:graph-envelope}.
\end{proof}

\subsection{Envelopes for a nonmonotone quantile statistic}
\label{subsec:quantile-envelope}

Graph inclusion does not by itself order $T_D$: adding vertices can lower the
lower median of the degrees.  We therefore bound every possible retained
degree first and only then take the quantile.  For $i\in V_B^+$ define
\begin{align}
 d_{B,i}^-
 &=\#\{j\in V_B^-\setminus\{i\}:
          U_{ij}\le k_{n_B^+}(x_i,x_j)\},
 \label{eq:degree-lower}\\
 \shortintertext{and}
 d_{B,i}^+
 &=\#\{j\in V_B^+\setminus\{i\}:
          U_{ij}\le k_{n_B^-}(x_i,x_j)\}.
 \label{eq:degree-upper}
\end{align}
If $i\in V_\theta$ and $d_i(\theta)=\deg_{G_\theta}(i)$, then
\begin{equation}
 d_{B,i}^-\le d_i(\theta)\le d_{B,i}^+,
 \qquad \theta\in B.
 \label{eq:degree-envelope-pointwise}
\end{equation}
Call the vertices in $V_B^-$ mandatory and those in
$O_B=V_B^+\setminus V_B^-$ optional.  Set
\begin{equation}
 N_B^-=|V_B^-|,\qquad N_B^+=|V_B^+|.
 \label{eq:mandatory-optional}
\end{equation}
For $N_B^-\ge2$, define
\begin{align}
 \underline d_B
 &=\min\Bigl\{k\ge0:
      2\#\{i\in V_B^-:d_{B,i}^-\le k\}
      +\#\{i\in O_B:d_{B,i}^-\le k\}\ge N_B^-\Bigr\},
 \label{eq:median-lower}\\
 \shortintertext{and}
 \overline d_B
 &=\max\Bigl\{0\le k\le N_B^+-1:
      2\#\{i\in V_B^-:d_{B,i}^+\ge k\}
      +\#\{i\in O_B:d_{B,i}^+\ge k\}>N_B^-\Bigr\}.
 \label{eq:median-upper}
\end{align}
For $N_B^-\ge2$, set
\begin{align}
 L_B^D&=\Pi_{[0,2]}\frac{2\log(1+\underline d_B)}{\log N_B^+},&
 U_B^D&=\Pi_{[0,2]}\frac{2\log(1+\overline d_B)}{\log N_B^-},
 \label{eq:degree-stat-envelope}\\
 L_B^E&=\Pi_{[0,1]}\left(\frac{\log(|E_B^-|\vee1)}{\log N_B^+}-1\right),&
 U_B^E&=\Pi_{[0,1]}\left(\frac{\log(|E_B^+|\vee1)}{\log N_B^-}-1\right).
 \label{eq:edge-stat-envelope}
\end{align}
For the remaining sizes, define
\begin{equation}
 (L_B^D,U_B^D,L_B^E,U_B^E)=
 \begin{cases}
 (0,0,0,0),&N_B^+\le1,\\
 (0,2,0,1),&N_B^-\le1<N_B^+.
 \end{cases}
 \label{eq:small-size-envelopes}
\end{equation}

\begin{lemma}
\label{lem:quantile-envelope}
For every $\theta\in B$,
\begin{equation}
 L_B^D\le T_D(G_\theta)\le U_B^D,
 \qquad
 L_B^E\le T_E(G_\theta)\le U_B^E.
 \label{eq:stat-envelope}
\end{equation}
\end{lemma}

\begin{proof}
Assume first that $N_B^-\ge2$.  For a candidate $k$, adding every optional
vertex whose lower degree bound is at most $k$ and no other optional vertex is
the most favorable choice for making the lower median at most $k$.  Hence
\begin{align}
 &\exists V:\ V_B^-\subseteq V\subseteq V_B^+:\
 \med_{i\in V} d_{B,i}^-\le k
 \notag\\
 &\qquad\Longleftrightarrow\quad
 2\#\{i\in V_B^-:d_{B,i}^-\le k\}
 +\#\{i\in O_B:d_{B,i}^-\le k\}\ge N_B^-.
 \label{eq:lower-median-count}
\end{align}
Similarly, retaining exactly the optional vertices whose upper degree bound is
at least $k$ is the most favorable choice for making the lower median at least
$k$, and
\begin{align}
 &\exists V:\ V_B^-\subseteq V\subseteq V_B^+:\
 \med_{i\in V} d_{B,i}^+\ge k
 \notag\\
 &\qquad\Longleftrightarrow\quad
 2\#\{i\in V_B^-:d_{B,i}^+\ge k\}
 +\#\{i\in O_B:d_{B,i}^+\ge k\}>N_B^-.
 \label{eq:upper-median-count}
\end{align}
Thus \eqref{eq:median-lower}--\eqref{eq:median-upper} are equivalently
\begin{equation}
 \underline d_B
 =\min_{V_B^-\subseteq V\subseteq V_B^+}\med_{i\in V}d_{B,i}^-,
 \qquad
 \overline d_B
 =\max_{V_B^-\subseteq V\subseteq V_B^+}\med_{i\in V}d_{B,i}^+.
 \label{eq:median-extrema-over-subsets}
\end{equation}
Taking $V=V_\theta$ in \eqref{eq:median-extrema-over-subsets} and using
\eqref{eq:degree-envelope-pointwise} gives
\begin{equation}
 \underline d_B\le \med_{i\in V_\theta} d_i(\theta)\le \overline d_B.
 \label{eq:median-envelope-proof}
\end{equation}
The common graph coupling also gives
\begin{equation}
 N_B^-\le N(G_\theta)\le N_B^+,
 \qquad
 |E_B^-|\le E(G_\theta)\le |E_B^+|.
 \label{eq:size-edge-envelope-proof}
\end{equation}
Equations~\eqref{eq:median-envelope-proof}--\eqref{eq:size-edge-envelope-proof},
together with the monotonicity of the logarithms and interval projection,
give \eqref{eq:stat-envelope}.  The cases $N_B^-\le1$ follow directly from
\eqref{eq:small-size-envelopes} and the ranges $0\le T_D\le2$ and
$0\le T_E\le1$.
\end{proof}

For the fixed statistic $T$, write $(L_{B,r},U_{B,r})$ for the corresponding
degree or edge envelope in replica $r$.  The optional subset attaining an
envelope need not arise from one common parameter value; only the outer bounds
are used below.

\subsection{A common simulation library for the entire search}
\label{subsec:common-ranks}

The statistic envelopes become inferentially useful by comparison with the
exact, generally uncomputed, Monte Carlo rank at the true parameter.  For
$r=1,\ldots,R$, take independent copies of the common Gaussian, Poisson, and
edge-mark construction; all rectangles within one copy reuse the same random
objects.  The observed graph is independent of these copies.  Put $T_0=T(G)$
and, for fixed $\theta$,
\begin{align}
 T_r(\theta)&=T(G_\theta^{(r)}),&
 q_\theta^\pm
 &=\frac{1+\sum_{r=1}^R
          \one\{\pm T_r(\theta)\ge\pm T_0\}}{R+1},
 \label{eq:true-parameter-rank-one-sided}\\
 \shortintertext{and}
 q_\theta&=\min\{1,2\min(q_\theta^+,q_\theta^-)\}.
 \label{eq:true-parameter-rank}
\end{align}

\begin{lemma}
\label{lem:rank-domination}
For every rectangle $B\ni\theta$,
\begin{equation}
 p_B^\pm(T_0)\ge q_\theta^\pm,
 \qquad
 p_B(T_0)\ge q_\theta.
 \label{eq:rank-domination}
\end{equation}
Under the joint law $Q_{\theta,h}\otimes P_h^{\rm aux}$,
\begin{equation}
 (Q_{\theta,h}\otimes P_h^{\rm aux})\{q_\theta^\pm\le u\}\le u,
 \qquad
 (Q_{\theta,h}\otimes P_h^{\rm aux})\{q_\theta\le u\}\le u,
 \qquad 0\le u\le1.
 \label{eq:true-rank-validity}
\end{equation}
\end{lemma}

\begin{proof}
For every replica $r$ and every $B\ni\theta$,
\begin{align}
 \one\{U_{B,r}\ge T_0\}
 &\ge\one\{T_r(\theta)\ge T_0\},
 \label{eq:upper-rank-indicator}\\
 \one\{L_{B,r}\le T_0\}
 &\ge\one\{T_r(\theta)\le T_0\}.
 \label{eq:lower-rank-indicator}
\end{align}
Thus $p_B^\pm(T_0)\ge q_\theta^\pm$, and coordinatewise monotonicity of
$(x,y)\mapsto\min\{1,2\min(x,y)\}$ gives the second inequality in
\eqref{eq:rank-domination}.

Under $Q_{\theta,h}\otimes P_h^{\rm aux}$, the variables
$T_0,T_1(\theta),\ldots,T_R(\theta)$ are independent and identically
distributed.  Proposition~\ref{prop:rank} gives the first bound in
\eqref{eq:true-rank-validity}.  Moreover, for $0\le u\le1$,
\begin{align}
 (Q_{\theta,h}\otimes P_h^{\rm aux})\{q_\theta\le u\}
 &\le (Q_{\theta,h}\otimes P_h^{\rm aux})\{q_\theta^+\le u/2\}
      +(Q_{\theta,h}\otimes P_h^{\rm aux})\{q_\theta^-\le u/2\}
 \notag\\
 &\le u.
 \label{eq:two-sided-rank-validity}
\end{align}
\end{proof}

Thus one event, $\{q_\theta\le\alpha\}$, controls every box containing the
true parameter, even when the boxes are selected adaptively.

\subsection{Refinement, projection, and stopping}
\label{subsec:refinement}

We now turn the boxwise domination in Lemma~\ref{lem:rank-domination} into an
adaptive search.  Fix a countable rectangle tree.  A split replaces a parent
by finitely many children whose union is the parent.  For every child
$B'\subset B$, intersect its newly computed bounds with the bounds inherited
from each ancestor by repeated updates
\begin{equation}
 L_{B',r}\leftarrow\max\{L_{B',r},L_{B,r}\},
 \qquad
 U_{B',r}\leftarrow\min\{U_{B',r},U_{B,r}\}.
 \label{eq:ancestor-intersection}
\end{equation}
Every updated interval still contains $T_r(\theta)$ for all $\theta\in B'$.
Consequently,
\begin{equation}
 p_{B'}^\pm(t)\le p_B^\pm(t),
 \qquad
 p_{B'}(t)\le p_B(t),
 \qquad B'\subset B,\quad t\in\R.
 \label{eq:refinement-monotonicity}
\end{equation}
A split preserves the current union before deletions, whereas discarding a box
can only shrink it.  Hence the retained sets satisfy
\begin{equation}
 \Theta_{j+1}\subseteq \Theta_j.
 \label{eq:retained-set-monotonicity}
\end{equation}
Unprocessed rectangles remain in $\Theta_j$, so the computation may stop at any
prescribed depth, work budget, or certified parameter statement.

For projection to the target parameter, define
\begin{equation}
 \Theta_j^\nu=\{\gamma^2/2:(\gamma,n)\in\Theta_j\}.
 \label{eq:nu-projection}
\end{equation}
Its smallest containing interval may be reported together with $\Theta_j^\nu$;
an empty set remains empty.

For a composite null region $\Theta_0\subset\Theta$, let $\mathcal B_0$ be a
finite rectangle cover.  At each level let $\mathcal B_j$ be a finite cover of
$\Theta_0$ obtained from $\mathcal B_0$ by adaptive splits using the fixed mother
simulations, and define
\begin{equation}
 p_{\Theta_0,j}^+(T_0)
 =\max_{B\in\mathcal B_j}p_B^+(T_0).
 \label{eq:composite-pvalue}
\end{equation}
Equation~\eqref{eq:refinement-monotonicity} gives
\begin{equation}
 p_{\Theta_0,j+1}^+(T_0)\le p_{\Theta_0,j}^+(T_0).
 \label{eq:composite-pvalue-monotonicity}
\end{equation}
Rectangles crossing the boundary of $\Theta_0$ only enlarge the null region
and therefore remain conservative.

\begin{corollary}
\label{cor:composite}
For every such adaptive refinement,
\begin{equation}
 \sup_{\theta\in\Theta_0}
 (Q_{\theta,h}\otimes P_h^{\rm aux})
 \left\{\exists j\ge0:
 p_{\Theta_0,j}^+(T_0)\le\alpha\right\}
 \le\alpha.
 \label{eq:composite-level}
\end{equation}
\end{corollary}

\begin{proof}
Fix $\theta\in\Theta_0$.  Since $\mathcal B_j$ covers $\Theta_0$, choose
$B_j\in\mathcal B_j$ with $\theta\in B_j$.  Lemma~\ref{lem:rank-domination}
gives, for every $j$,
\begin{equation}
 p_{\Theta_0,j}^+(T_0)\ge p_{B_j}^+(T_0)\ge q_\theta^+.
 \label{eq:composite-rank-domination}
\end{equation}
Therefore
\[
 \left\{\exists j\ge0:p_{\Theta_0,j}^+(T_0)\le\alpha\right\}
 \subseteq\{q_\theta^+\le\alpha\},
\]
and \eqref{eq:true-rank-validity} proves \eqref{eq:composite-level}.
\end{proof}

\subsection{Local exactness of the envelopes}
\label{subsec:local-exactness}

Coverage uses only outer envelopes, but refinement is useful because their
extra width disappears locally at fixed simulator resolution.  Fix $h,R$ and a
deterministic $\theta\in\Theta$.

\begin{lemma}
\label{lem:local-exactness}
\[
 P_h^{\rm aux}\!\left\{\exists U\subset\Theta:\
 \theta\in U,\ U\text{ relatively open},\
 G_\vartheta^{(r)}=G_\theta^{(r)}
 \ \forall\vartheta\in U,\ 1\le r\le R\right\}=1.
\]
Consequently, for every nested sequence $B_k\ni\theta$ with
$B_{k+1}\subseteq B_k$ and $\diam B_k\to0$,
\begin{equation}
 P_h^{\rm aux}\!\left\{\exists k_0<\infty:\
 L_{B_k,r}=U_{B_k,r}=T_r(\theta)
 \ \forall k\ge k_0,\ 1\le r\le R\right\}=1.
 \label{eq:local-exactness}
\end{equation}
\end{lemma}

\begin{proof}
Conditional on the $R$ finite Gaussian vectors, all global mother sets are
finite.  At the fixed deterministic $\theta$, a mother-node height equals its
retention threshold with probability zero, and an independent uniform edge
mark equals its edge threshold with probability zero.  Since only finitely
many node and edge comparisons occur, almost surely every comparison is
strict.  Their threshold functions are continuous in $(\gamma,n)$, so all
comparison signs are unchanged on some relative neighborhood of $\theta$.
Thus every simulated graph is constant on that neighborhood.

For $B_k\ni\theta$ with $\diam B_k\to0$,
\begin{equation}
 \lambda_{B_k}^\pm(x)\longrightarrow n w_{\gamma,h}(x),
 \qquad
 k_{n_{B_k}^\pm}(x,y)\longrightarrow k_n(x,y).
 \label{eq:shrinking-envelope-thresholds}
\end{equation}
The strict comparisons above therefore imply that, for all sufficiently large
$k$, every mother node and every possible edge has the same status in the two
envelope graphs and in $G_\theta$.  Hence $G_{B_k}^-=G_\theta=G_{B_k}^+$.
If the common graph has at least two vertices,
\eqref{eq:median-lower}--\eqref{eq:edge-stat-envelope} are then exact; if it
has zero or one vertex, \eqref{eq:small-size-envelopes} is exact.  This proves
\eqref{eq:local-exactness}.
\end{proof}

\subsection{Joint size--degree residual calibration}
\label{subsec:residual}

Part~(ii) of Theorem~\ref{thm:calibration} uses an independent pilot sample to
freeze a boxwise score and a fresh simulation library to calibrate that score.
Write $L_N=\log(N\vee2)$ for the first coordinate of $S(G)$.  The random
size $N$ varies with the total environment mass in \eqref{eq:random-size}, so
$L_N$ can explain part of the variation in $T_D(G)$.  If $L_N$ and $T_D(G)$
are square-integrable and $\Var(L_N)>0$, then, for every $\xi\in\R$,
\begin{equation}
 \Var\bigl(T_D(G)-\xi L_N\bigr)
 =\Var\bigl(T_D(G)\bigr)
  -\frac{\Cov(T_D(G),L_N)^2}{\Var(L_N)}
  +\Var(L_N)\left(\xi-\frac{\Cov(T_D(G),L_N)}{\Var(L_N)}\right)^2.
 \label{eq:linear-residual-identity}
\end{equation}
This motivates the affine residual direction in \eqref{eq:residual-score}; its
full finite-sample distribution is calibrated rather than approximated by a
Gaussian law.  The separation of score construction and calibration follows
the simulator-based inference framework of \cite[Sections~3.2--3.4]{LF2I}.

Fix the finite partition $\mathcal B$ and draw $R_0\ge2$ independent pilot
mother simulations.  For each $B\in\mathcal B$, let
$[\ell^0_{B,r},u^0_{B,r}]$ be the pilot envelope for $S$ and define
\begin{align}
 z^0_{B,r}&=\tfrac12(\ell^0_{B,r}+u^0_{B,r}),&
 m_B&=\frac1{R_0}\sum_{r=1}^{R_0}z^0_{B,r},
 \label{eq:pilot-center}\\
 \shortintertext{and}
 \widehat\Sigma_B&=\frac1{R_0-1}\sum_{r=1}^{R_0}
       (z^0_{B,r}-m_B)(z^0_{B,r}-m_B)^\top.
 \label{eq:pilot-covariance}
\end{align}
Choose $0<\lambda<1$, $\varepsilon>0$, and $e_1,e_2>0$ before sampling, and set
\begin{equation}
 \Sigma_B=(1-\lambda)\widehat\Sigma_B
  +\lambda\diag\bigl(\max\{(\widehat\Sigma_B)_{11},e_1\},
                     \max\{(\widehat\Sigma_B)_{22},e_2\}\bigr)+\varepsilon I_2.
 \label{eq:pilot-regularization}
\end{equation}
Then
\begin{equation}
 \Sigma_B\succeq\varepsilon I_2
 \quad\Longrightarrow\quad
 (\Sigma_B)_{11}>0,\qquad \sigma_B^2>0,
 \label{eq:pilot-positive-definite}
\end{equation}
so \eqref{eq:residual-coefficients}--\eqref{eq:residual-score} define
$(a_B,\rho_B)$ for every pilot realization.

For a fresh calibration replica $r$, denote its envelope quantities by
$N_{B,r}^\pm$, $L_{B,r}^D$, and $U_{B,r}^D$.  Then
\begin{equation}
 \ell_{B,r}=
 \begin{pmatrix}\log(N_{B,r}^-\vee2)\\L_{B,r}^D\end{pmatrix},
 \qquad
 u_{B,r}=
 \begin{pmatrix}\log(N_{B,r}^+\vee2)\\U_{B,r}^D\end{pmatrix}.
 \label{eq:joint-feature-envelope}
\end{equation}
Write
\begin{equation}
 z_{B,r}=\tfrac12(\ell_{B,r}+u_{B,r}),
 \qquad
 \Delta z_{B,r}=\tfrac12(u_{B,r}-\ell_{B,r}),
 \label{eq:joint-feature-center-width}
\end{equation}
where absolute values of vectors below are coordinatewise.

\begin{lemma}
\label{lem:affine-envelope}
The exact extrema of $|a_B^\top(z-m_B)|$ on
$[\ell_{B,r},u_{B,r}]$ are
\begin{align}
 L^\rho_{B,r}
 &=\max\{0,|a_B^\top(z_{B,r}-m_B)|-|a_B|^\top\Delta z_{B,r}\},
 \label{eq:affine-lower}\\
 \shortintertext{and}
 U^\rho_{B,r}
 &=|a_B^\top(z_{B,r}-m_B)|+|a_B|^\top\Delta z_{B,r}.
 \label{eq:affine-upper}
\end{align}
Consequently, for every $\theta\in B$,
\begin{equation}
 L^\rho_{B,r}\le\rho_B(G_\theta^{(r)})\le U^\rho_{B,r}.
 \label{eq:affine-score-envelope}
\end{equation}
\end{lemma}

\begin{proof}
The linear image of the feature rectangle is
\begin{align}
 \{a_B^\top(z-m_B):z\in[\ell_{B,r},u_{B,r}]\}
 =\bigl[&a_B^\top(z_{B,r}-m_B)-|a_B|^\top\Delta z_{B,r},\notag\\
        &a_B^\top(z_{B,r}-m_B)+|a_B|^\top\Delta z_{B,r}\bigr].
 \label{eq:affine-linear-image}
\end{align}
Taking the minimum and maximum absolute values gives
\eqref{eq:affine-lower}--\eqref{eq:affine-upper}.  Since
\eqref{eq:joint-feature-envelope} contains every feasible
$S(G_\theta^{(r)})$, \eqref{eq:affine-score-envelope} follows.
\end{proof}

With the pilot fixed, the independent calibration replicas use the upper
envelope $U^\rho_{B,r}$ in \eqref{eq:residual-calibration}.  Because the score
$\rho_B$ itself depends on $B$, all boxes of the prespecified partition are
calibrated separately.

For the degree and edge procedures, write
\begin{equation}
 p_B^D(G)=p_B(T_D(G)),
 \qquad
 p_B^E(G)=p_B(T_E(G)).
 \label{eq:degree-edge-box-scores}
\end{equation}
Fix one score family among $p_B^D,p_B^E,p_B^\rho$ and denote it by $p_B$.
Its deterministic point estimate is
\begin{equation}
 \mathcal A(G)=\argmax_{B\in\mathcal B}p_B(G),
 \qquad
 \widehat\nu_{\mathrm{cal}}(G)
 =\frac1{|\mathcal A(G)|}\sum_{B\in\mathcal A(G)}
       \frac{\nu_-(B)+\nu_+(B)}2,
 \label{eq:calibrated-point-estimate}
\end{equation}
where $\nu_\pm(B)=\gamma_\pm(B)^2/2$.

\subsection{Final proof of Theorem~\ref{thm:calibration}}
\label{subsec:final-calibration}

For part~(i), Lemmas~\ref{lem:graph-envelope} and
\ref{lem:quantile-envelope} give \eqref{eq:envelope-intro} simultaneously for
every rectangle in the countable tree.  Fix the true parameter $\theta$.  If
$\theta$ leaves the retained union at any step, then some rectangle
$B\ni\theta$ has been discarded and hence $p_B(T_0)\le\alpha$.  By
Lemma~\ref{lem:rank-domination},
\begin{equation}
 \{\exists j\ge0:\theta\notin\Theta_j\}
 \subseteq\{q_\theta\le\alpha\}.
 \label{eq:single-error-event}
\end{equation}
Equation~\eqref{eq:true-rank-validity} therefore gives
\eqref{eq:main-calibration}.  The inclusion is pathwise, so the selected boxes
and stopping step may depend on the data and all previously computed
envelopes.  Equation~\eqref{eq:retained-set-monotonicity} proves
$\Theta_{j+1}\subseteq \Theta_j$, and Lemma~\ref{lem:local-exactness} proves the local
exactness assertion.  Finally, $\theta\in\Theta_j$ implies
$\gamma^2/2\in\Theta_j^\nu$, so projection preserves coverage.

For part~(ii), condition on the pilot $\mathcal D_0$ and fix $\theta\in B$.
By independence of the target graph and the calibration replicas,
\[
 \rho_B(G),\rho_B(G_\theta^{(1)}),\ldots,
 \rho_B(G_\theta^{(R)})
\]
are conditionally independent and identically distributed.  Define
\begin{equation}
 q_{\theta,B}^\rho
 =\frac{1+\sum_{r=1}^R
   \one\{\rho_B(G_\theta^{(r)})\ge\rho_B(G)\}}{R+1}.
 \label{eq:residual-exact-rank}
\end{equation}
Lemma~\ref{lem:affine-envelope} gives
\begin{equation}
 p_B^\rho(G)\ge q_{\theta,B}^\rho.
 \label{eq:residual-rank-domination}
\end{equation}
Proposition~\ref{prop:rank}, applied conditionally on $\mathcal D_0$, proves
\eqref{eq:main-residual-size}.

Choose a deterministic box $B_*(\theta)\in\mathcal B$ containing $\theta$,
using a fixed ordering on shared boundaries.  From
\eqref{eq:residual-confidence-set},
\begin{equation}
 \{\gamma^2/2\notin C_\alpha^\rho(G)\}
 \subseteq
 \{p_{B_*(\theta)}^\rho(G)\le\alpha\}.
 \label{eq:residual-coverage-inclusion}
\end{equation}
Combining \eqref{eq:residual-coverage-inclusion} with
\eqref{eq:main-residual-size} proves \eqref{eq:main-residual-coverage}.
Equation~\eqref{eq:pilot-positive-definite} gives $\sigma_B^2>0$ for every
pilot realization.

\section{Graph-law approximation and the continuous model}
\label{sec:continuum}

This section transfers finite-resolution graph laws to the continuous GMC
model.  We first compare regular Cox laws by covariance interpolation and then
show that weak convergence of directing measures becomes total variation after
positions and labels are erased.  The periodic FFT--pixel approximation is
treated separately to obtain uniform quantitative rates.

\subsection{Gaussian interpolation and Poisson differences}
\label{subsec:interpolation}

We begin with a regular covariance comparison that will later be applied before
passing to the singular GMC limit.  Let $K_0,K_1$ be bounded continuous
covariances on the compact observation domain, or finite-pixel covariances, and
let $X_j$ be centered Gaussian fields with covariance $K_j$, $j=0,1$.  Let
$\Xi_j$ be the Cox point process with conditional intensity
\begin{equation}
 \lambda_j(x)\dd x
 =n e^{\gamma X_j(x)-\gamma^2K_j(x,x)/2}\dd x,
 \qquad j=0,1,
 \label{eq:regular-cox-intensity}
\end{equation}
and let $Q_j$ be its image under the same measurable graph-generation kernel.
Write $\EE_j$ for expectation at the endpoint $K_j$, including Cox and graph
randomness when present.

\begin{lemma}
\label{lem:regular-tv}
\begin{equation}
 \max\left\{
 d_{\TV}\bigl(\Law(\Xi_0),\Law(\Xi_1)\bigr),
 d_{\TV}(Q_0,Q_1)
 \right\}
 \le1\wedge n^2\iint_{D^2}
       \left|e^{\gamma^2K_1(x,y)}-e^{\gamma^2K_0(x,y)}\right|\dd x\dd y.
 \label{eq:regular-tv}
\end{equation}
\end{lemma}

\begin{proof}
First use a finite partition $D=\bigsqcup_{i\in I}D_i$ on which both
covariances are constant.  Let $\Sigma_j$ be the corresponding cell covariance
matrix after absorbing $\gamma$, so that
\begin{align}
 (\Sigma_j)_{ik}&=\gamma^2K_j(x,y),
 &&x\in D_i,\ y\in D_k,\quad j=0,1,\notag\\
 \Sigma_t&=(1-t)\Sigma_0+t\Sigma_1,
 &\dot\Sigma&=\Sigma_1-\Sigma_0,\notag\\
 Y_t&\sim N(0,\Sigma_t),
 &0&\le t\le1.
 \label{eq:interpolation-covariance}
\end{align}
Put
\begin{equation}
 \lambda_i(t,y)=n|D_i|e^{y_i-(\Sigma_t)_{ii}/2}.
 \label{eq:interpolation-intensity}
\end{equation}
For an event concerning the point process or the graph, condition on the cell
counts and average over positions within the cells and over graph marks.  This
gives a fixed $H:\N_0^I\to[0,1]$.  Write
\[
 \nabla_i^+H(z)=H(z+e_i)-H(z),
 \qquad
 \nabla_i^+\nabla_j^+H=\nabla_i^+(\nabla_j^+H),
 \qquad
 Z_i(t,y)\sim\operatorname{Pois}(\lambda_i(t,y)),
\]
with the $Z_i(t,y)$ conditionally independent, and set
\[
 \Phi_t(y)=\EE H\bigl((Z_i(t,y))_{i\in I}\bigr).
\]
Termwise differentiation of the Poisson series gives
\begin{align}
 \partial_i\Phi_t
 &=\lambda_i\EE[\nabla_i^+H\mid y],
 \label{eq:first-poisson-derivative}\\
 \partial_{ij}\Phi_t
 &=\one\{i=j\}\lambda_i\EE[\nabla_i^+H\mid y]
       +\lambda_i\lambda_j\EE[\nabla_i^+\nabla_j^+H\mid y],
 \label{eq:second-poisson-derivative}\\
 \partial_t\Phi_t
 &=-\tfrac12\sum_i\dot\Sigma_{ii}\partial_i\Phi_t.
 \label{eq:variance-normalization-derivative}
\end{align}
Gaussian integration by parts yields
\begin{equation}
 \frac{\dd}{\dd t}\EE\Phi_t(Y_t)
 =\EE\partial_t\Phi_t(Y_t)
      +\tfrac12\sum_{i,j}\dot\Sigma_{ij}\EE\partial_{ij}\Phi_t(Y_t).
 \label{eq:gaussian-interpolation-step}
\end{equation}
The diagonal term in \eqref{eq:second-poisson-derivative} cancels
\eqref{eq:variance-normalization-derivative}, so
\begin{equation}
 \frac{\dd}{\dd t}\EE H
 =\tfrac12\sum_{i,j}\dot\Sigma_{ij}
    \EE\left[\lambda_i\lambda_j\EE(\nabla_i^+\nabla_j^+H\mid Y_t)\right].
 \label{eq:cox-cancellation}
\end{equation}
Since
\[
 |\nabla_i^+\nabla_j^+H|\le2,
 \qquad
 \EE(\lambda_i\lambda_j)=n^2|D_i||D_j|e^{(\Sigma_t)_{ij}},
\]
integration in $t$ gives
\begin{align}
 |\EE_1H-\EE_0H|
 &\le n^2\sum_{i,j}|D_i||D_j|
       \int_0^1|\dot\Sigma_{ij}|e^{(\Sigma_t)_{ij}}\dd t
 \label{eq:finite-covariance-bound-a}\\
 &=n^2\sum_{i,j}|D_i||D_j|
       |e^{(\Sigma_1)_{ij}}-e^{(\Sigma_0)_{ij}}|.
 \label{eq:finite-covariance-bound}
\end{align}
Degenerate covariance matrices are handled by adding the same positive
multiple of the identity at both endpoints and sending it to zero.

For bounded continuous covariances, let $z_m(x)$ be the center of the cell
containing $x$, and define
\begin{equation}
 X_m(x)=X(z_m(x)),
 \qquad
 W_m(x)=e^{\gamma X_m(x)-\gamma^2\EE X_m(x)^2/2},
 \qquad
 W(x)=e^{\gamma X(x)-\gamma^2\EE X(x)^2/2}.
 \label{eq:cell-center-field}
\end{equation}  Then
\begin{equation}
 \EE|W_m(x)-W(x)|^2
 =e^{\gamma^2K(z_m(x),z_m(x))}+e^{\gamma^2K(x,x)}
       -2e^{\gamma^2K(z_m(x),x)}\longrightarrow0.
 \label{eq:cell-center-l2}
\end{equation}
Bounded point variances and dominated convergence imply
\[
 \EE\int_D|W_m-W|\dd x\longrightarrow0.
\]
Let $\Xi^{(m)}$ and $\Xi$ have conditional intensities $nW_m\dd x$ and
$nW\dd x$.  Common Poisson thinning gives
\begin{equation}
 d_{\TV}\bigl(\Law(\Xi^{(m)}),\Law(\Xi)\bigr)
 \le n\EE\int_D|W_m-W|\dd x\longrightarrow0.
 \label{eq:cell-center-cox-tv}
\end{equation}
The right side of \eqref{eq:finite-covariance-bound} converges by dominated
convergence.  Pixel covariances are already covered by a common refinement of
their partitions, and a common graph-generation Markov kernel contracts total
variation.  Taking the supremum over events proves \eqref{eq:regular-tv}.
\end{proof}

\subsection{A deletion formula for decision sensitivity}
\label{subsec:deletion}

The estimate in Lemma~\ref{lem:regular-tv} takes absolute values before
averaging and therefore discards cancellation.  For a fixed graph decision,
the same interpolation admits a signed vertex-deletion representation.  Fix
$\gamma,n$ and the independent-edge connection rule, and set
\begin{equation}
 K_t=(1-t)K_0+tK_1,
 \qquad
 \Delta K=K_1-K_0,
 \qquad 0\le t\le1,
 \label{eq:deletion-interpolation}
\end{equation}
where $K_0,K_1$ are bounded continuous covariances or are constant on a common
finite partition.  Write $\EE_t$ for expectation under the interpolated Cox
and graph law.  For a bounded measurable graph function $0\le H\le1$, define
\begin{equation}
 \nabla^-_{ij}H(G)
 =H(G)-H(G-i)-H(G-j)+H(G-\{i,j\}),
 \label{eq:deletion-difference}
\end{equation}
where $G-i$ is the induced graph obtained by deleting vertex $i$.  The four
graphs use the same retained edge marks, and any fitted quantities entering
$H$ are held fixed.  For a realized Cox configuration, write $(x_i)$ for the
latent vertex locations.

\begin{lemma}
\label{lem:deletion-identity}
\begin{equation}
 \EE_1H(G)-\EE_0H(G)
 =\gamma^2\int_0^1\EE_t\sum_{i<j}
     \Delta K(x_i,x_j)\nabla^-_{ij}H(G)\dd t.
 \label{eq:deletion-identity}
\end{equation}
\end{lemma}

\begin{proof}
Put
\begin{equation}
 \lambda_t(x)=n e^{\gamma X_t(x)-\gamma^2K_t(x,x)/2}.
 \label{eq:deletion-intensity}
\end{equation}
Write $G\oplus x$ for the graph obtained by adding a vertex at $x$ with
fresh independent incident edge marks, and define $G\oplus x\oplus y$
successively.  Then
\begin{equation}
 \nabla_{x,y}^+H(G)
 =H(G\oplus x\oplus y)-H(G\oplus x)-H(G\oplus y)+H(G).
 \label{eq:add-two-difference}
\end{equation}
For a finite partition, \eqref{eq:cox-cancellation} becomes
\begin{equation}
 \frac{\dd}{\dd t}\EE_tH(G)
 =\frac{\gamma^2}{2}\iint\Delta K(x,y)
   \EE_t\!\left[\lambda_t(x)\lambda_t(y)
               \EE_t[\nabla_{x,y}^+H(G)\mid X_t]\right]\dd x\dd y.
 \label{eq:deletion-add-form}
\end{equation}
Retaining all original edge marks, the twice-iterated Mecke identity gives
\begin{align}
 &\iint\Delta K(x,y)
   \EE_t\!\left[\lambda_t(x)\lambda_t(y)
               \EE_t[\nabla_{x,y}^+H(G)\mid X_t]\right]\dd x\dd y \notag\\
 &\hspace{4em}=\EE_t\sum_{i\ne j}
          \Delta K(x_i,x_j)\nabla^-_{ij}H(G).
 \label{eq:deletion-mecke}
\end{align}
The ordered sum on the right is twice the sum over $i<j$; integrating
\eqref{eq:deletion-add-form} proves \eqref{eq:deletion-identity} for pixel
covariances.  The first-order Poisson perturbation identity underlying this
step is \cite[Corollary~6.7, author version]{LastPerturbation}.

For bounded continuous covariances, approximate on a common partition as in
Lemma~\ref{lem:regular-tv}.  Uniformly bounded point variances give
\[
 \sup_t\EE_tN^4<\infty,
 \qquad
 \left|\sum_{i<j}\Delta K(x_i,x_j)\nabla^-_{ij}H(G)\right|
 \le2\norm{\Delta K}_\infty\binom N2.
\]
Hence the deletion sums are uniformly integrable.  The common-thinning
coupling and uniform covariance approximation give convergence in probability,
and integration in $t$ yields \eqref{eq:deletion-identity}.
\end{proof}

Although $(x_i)$ are latent simulator coordinates, $H$ depends only on the
observed graph.  Taking the absolute value after the signed sum gives
\begin{equation}
 |\EE_1H-\EE_0H|
 \le\gamma^2\int_0^1
       \EE_t\left|\sum_{i<j}\Delta K(x_i,x_j)
                            \nabla^-_{ij}H(G)\right|\dd t.
 \label{eq:deletion-bound}
\end{equation}

For the residual score $\rho_B(G)=|a_B^\top(S(G)-m_B)|$, we next bound the
response to two deletions using only the observed degrees.  Let $N\ge4$,
$k=\lceil N/2\rceil$, and let
$d_{(1)}\le\cdots\le d_{(N)}$ be the ordered degrees of $G$.  If $G'$ is
obtained by deleting $q\in\{0,1,2\}$ vertices, put
\begin{equation}
 N'=|V(G')|=N-q,
 \qquad
 j_q=\left\lceil\frac{N'}2\right\rceil,
 \qquad
 k-1\le j_q\le j_q+q\le k+1.
 \label{eq:deletion-ranks}
\end{equation}
Let
\[
 \widetilde d_{(1)}\le\cdots\le\widetilde d_{(N')},
 \qquad
 d'_{(1)}\le\cdots\le d'_{(N')},
\]
be, respectively, the original degrees of the retained vertices and their
degrees in $G'$.  Then, for $1\le r\le N'$,
\begin{equation}
 d_{(r)}\le\widetilde d_{(r)}\le d_{(r+q)},
 \qquad
 (\widetilde d_{(r)}-q)_+\le d'_{(r)}\le\widetilde d_{(r)},
 \qquad
 \med d(G')=d'_{(j_q)}.
 \label{eq:deletion-retained-order}
\end{equation}
Therefore
\begin{equation}
 N-2\le N'\le N,
 \qquad
 (d_{(k-1)}-2)_+\le\med d(G')\le d_{(k+1)}.
 \label{eq:deletion-median}
\end{equation}

Consequently every such $S(G')$ belongs to the rectangle with endpoints
\begin{align}
 \ell^-(G)&=\left(\log(N-2),\,
       \frac{2\log(1+(d_{(k-1)}-2)_+)}{\log N}\right)^\top,\notag\\
 u^-(G)&=\left(\log N,\,
       \Pi_{[0,2]}\frac{2\log(1+d_{(k+1)})}{\log(N-2)}\right)^\top.
 \label{eq:deletion-feature-box}
\end{align}
Define
\begin{align}
 L_B^{\rm del}(G)
 &=\left(
 \left|a_B^\top\left(\frac{\ell^-(G)+u^-(G)}2-m_B\right)\right|
 -\frac12|a_B|^\top\{u^-(G)-\ell^-(G)\}
 \right)_+,
 \notag\\
 U_B^{\rm del}(G)
 &=\left|a_B^\top\left(\frac{\ell^-(G)+u^-(G)}2-m_B\right)\right|
 +\frac12|a_B|^\top\{u^-(G)-\ell^-(G)\}.
 \label{eq:deletion-score-box}
\end{align}
For $H_z(G)=\one\{\rho_B(G)>z\}$,
\begin{equation}
 |\nabla^-_{ij}H_z(G)|
 \le2\one\{L_B^{\rm del}(G)\le z<U_B^{\rm del}(G)\},
 \qquad N\ge4.
 \label{eq:deletion-pivotal}
\end{equation}
For $N<4$, use simply
\[
 |\nabla^-_{ij}H_z|\le2.
\]

\subsection{Weak convergence of chaos and total variation of graphs}
\label{subsec:weak-to-tv}

Weak convergence of directing measures is enough for the spatial Cox process.
After positions and labels are erased, the countable graph state space upgrades
the conclusion to total variation.  Let $E$ be compact, let $\mu_j,\mu$ be
almost surely non-atomic random finite measures, and assume convergence in
distribution for the weak topology of finite measures,
\begin{equation}
 \mu_j\Longrightarrow\mu,
 \qquad
 n_j\longrightarrow n>0,
 \qquad
 \norm{k_j-k}_\infty\longrightarrow0,
 \label{eq:weak-to-tv-assumptions}
\end{equation}
where $k_j,k:E^2\to[0,1]$ are continuous and symmetric.  Let $\Xi_j,\Xi$ be
the Cox point processes directed by $n_j\mu_j,n\mu$, and let $Q_j,Q$ be the
corresponding unlabelled graph laws with isolated vertices retained.

\begin{lemma}
\label{lem:graph-tv-convergence}
\begin{equation}
 \Xi_j\Longrightarrow\Xi,
 \qquad
 d_{\TV}(Q_j,Q)\longrightarrow0.
 \label{eq:weak-to-graph-tv}
\end{equation}
\end{lemma}

\begin{proof}
For continuous $f\ge0$, the Cox Laplace functional is
\begin{equation}
 \EE\exp\left\{-n_j\int_E(1-e^{-f})\,\dd\mu_j\right\}.
 \label{eq:cox-laplace-functional}
\end{equation}
The map
\[
 (a,\lambda)\longmapsto
 \exp\left\{-a\int_E(1-e^{-f})\,\dd\lambda\right\}
\]
is bounded and continuous for $a>0$ and the weak topology on finite measures.
Thus \eqref{eq:weak-to-tv-assumptions} implies $\Xi_j\Rightarrow\Xi$.

Let $G_j\sim Q_j$, $G\sim Q$, and fix an unlabelled graph $g$ with
$m$ vertices.  Put $[m]=\{1,\ldots,m\}$.  For ordered positions
$x_1,\ldots,x_m$, define
\begin{equation}
 \Phi_{g,j}(x_1,\ldots,x_m)=
 \sum_{\substack{G'\text{ labelled on }[m]\\G'\simeq g}}
 \prod_{\{r,s\}\in E(G')}k_j(x_r,x_s)
 \prod_{\{r,s\}\in\binom{[m]}2\setminus E(G')}(1-k_j(x_r,x_s)).
 \label{eq:labelled-orbit}
\end{equation}
Let $\Phi_g$ denote the same expression with $k$ in place of $k_j$.  Then
\begin{equation}
 0\le\Phi_{g,j}\le1,
 \qquad
 \norm{\Phi_{g,j}-\Phi_g}_\infty\longrightarrow0.
 \label{eq:orbit-uniform-convergence}
\end{equation}
Writing $M_j=\mu_j(E)$, the Janossy formula gives
\begin{equation}
 \Pr(G_j\simeq g)
 =\frac{n_j^m}{m!}\EE\left[
 e^{-n_jM_j}\int_{E^m}\Phi_{g,j}(x_1,\ldots,x_m)
          \prod_{i=1}^m\mu_j(\dd x_i)\right].
 \label{eq:graph-pattern-probability}
\end{equation}
For $m=0$, the right side is $\EE e^{-n_jM_j}$.  For $m\ge1$,
\begin{equation}
 \frac{n_j^m e^{-n_jM_j}M_j^m}{m!}
 \le \sup_{a>0}\frac{a^me^{-a}}{m!}
 =\frac{m^me^{-m}}{m!}.
 \label{eq:janossy-uniform-bound}
\end{equation}
Hence replacing $\Phi_{g,j}$ by $\Phi_g$ in
\eqref{eq:graph-pattern-probability} changes the probability by at most
\[
 \norm{\Phi_{g,j}-\Phi_g}_\infty
 \frac{m^me^{-m}}{m!}\longrightarrow0.
\]
With $\Phi_g$ fixed, the remaining expression in
\eqref{eq:graph-pattern-probability} is a bounded continuous function of
$(n_j,\mu_j)$, so
\begin{equation}
 \Pr(G_j\simeq g)\longrightarrow\Pr(G\simeq g)
 \qquad\text{for every }g.
 \label{eq:graph-class-convergence}
\end{equation}
The graph isomorphism classes are countable and the limiting masses sum to one.
The discrete Scheff\'e lemma therefore gives
\begin{equation}
 \frac12\sum_g
 \left|\Pr(G_j\simeq g)-\Pr(G\simeq g)\right|
 \longrightarrow0,
 \label{eq:discrete-scheffe}
\end{equation}
which is \eqref{eq:weak-to-graph-tv}.
\end{proof}

For the quantitative transfer, fix $\gamma^2<2$, $n\ge1$, and a bounded
regular simulator covariance $K_h$ from the setup of
Theorem~\ref{thm:transfer}.  Let $\Xi_{\gamma,n}$ and
$\Xi_{\gamma,n,h}$ be the annealed continuous and simulator Cox point
processes, respectively.

\begin{lemma}
\label{lem:continuum-tv-bound}
\begin{equation}
 \begin{aligned}
 &d_{\TV}(P_{\gamma,n},Q_{\gamma,n,h})
 \vee d_{\TV}\bigl(\Law(\Xi_{\gamma,n}),\Law(\Xi_{\gamma,n,h})\bigr)\\
 &\qquad\le1\wedge n^2\iint_{D^2}
 \left|e^{\gamma^2K(x,y)}-e^{\gamma^2K_h(x,y)}\right|\dd x\dd y.
 \end{aligned}
 \label{eq:continuum-tv-bound}
\end{equation}
\end{lemma}

\begin{proof}
Compare an auxiliary regularization $K_\delta$ of the target field with $K_h$
using Lemma~\ref{lem:regular-tv}.  For the graph laws,
Lemma~\ref{lem:graph-tv-convergence} permits $\delta$ to tend to the continuum
limit on the left.  For the Cox laws, the regularized point-process laws
converge weakly on the Polish space of finite point measures, while total
variation from the fixed $K_h$ law is lower semicontinuous under weak
convergence.  Thus the same passage follows by a lower limit.

The approximation hypothesis before Theorem~\ref{thm:transfer} gives
\begin{equation}
 e^{\gamma^2K_\delta(x,y)}
 \le C\bigl(1+|x-y|^{-\gamma^2}\bigr),
 \label{eq:continuum-exponential-majorant}
\end{equation}
which is integrable on $D^2$ because $\gamma^2<2$.  The bounded covariance
$K_h$ contributes an integrable term as well.  Dominated convergence proves
\eqref{eq:continuum-tv-bound}.
\end{proof}

\subsection{The exact periodic FFT--pixel approximation}
\label{subsec:fft-pixel}

We first identify the ideal exact-arithmetic real-FFT sampler in
\eqref{eq:fft-indexing}--\eqref{eq:fft-chaos}; the next two lemmas verify the
hypotheses of Proposition~\ref{prop:gmc-approximation}.  The convention in
\eqref{eq:fft-indexing}--\eqref{eq:fft-chaos} sets the zero Fourier mode to
zero and uses real coefficients at self-conjugate Nyquist residues.  Write
$\operatorname{PPP}(\lambda)$ for the Poisson point-process law with intensity
measure $\lambda$.

\begin{lemma}
\label{lem:fft-identification}
\begin{equation}
 \Cov\{X_h(\pi_hx),X_h(\pi_hy)\}=K_h^{\T}(x,y),
 \qquad
 \Var\{X_h(\pi_hx)\}=v_h,
 \label{eq:fft-sampler-covariance}
\end{equation}
and
\begin{equation}
 \Law(\Xi_{\gamma,n,h}^{\T}\mid X_h)
 =\operatorname{PPP}(n\mu_{\gamma,h}^{\T}),
 \qquad
 n\mu_{\gamma,h}^{\T}(\dd x)
 =n e^{\gamma X_h(\pi_hx)-\gamma^2v_h/2}\dd x.
 \label{eq:fft-sampler-directing-measure}
\end{equation}
\end{lemma}

\begin{proof}
The functions $\{\phi_{h,r}:r\in\Lambda_h\}$ form an orthonormal basis of the
$h^2$-dimensional pixelwise constant subspace of $L^2(\T^2)$.  Under the
discrete Fourier transform, a real white-noise array has conjugate coefficients
at $r$ and $-r$; the self-conjugate residues have both coordinates in
$\{0,h/2\}$ and are real.  At a nonzero residue the sampler uses the multiplier
\begin{equation}
 m_h(r)=\frac{h}{\sqrt{2\pi|\widetilde r|^2}}
          e^{-9|\widetilde r|^2/(2h^2)}.
 \label{eq:fft-multiplier}
\end{equation}
The inverse transform contributes $h^{-2}$ and each white-noise Fourier
coefficient has variance $h^2$, hence
\[
 \frac{m_h(r)^2}{h^2}=\widehat K_h(r).
\]
Conjugate pairs and self-conjugate residues therefore give
\eqref{eq:fft-sampler-covariance}, with
\[
 v_h=\sum_{r\in\Lambda_h}\widehat K_h(r).
\]
The pixel density is consequently \eqref{eq:fft-chaos}; selecting a pixel with
probability proportional to this density and then adding independent uniform
jitter is Poisson sampling from \eqref{eq:fft-sampler-directing-measure}.
\end{proof}

Before the next lemma, set
\begin{equation}
 M_{\gamma,h}^{\T}=\mu_{\gamma,h}^{\T}(\T^2).
 \label{eq:fft-total-mass}
\end{equation}

\begin{lemma}
\label{lem:fft-admissibility}
\begin{align}
 &\lim_{\substack{h\to\infty\\h\ \text{even}}}
 \norm{K_h^{\T}(x,y)-K_{\T}(x-y)}_{L^2(\T^2\times\T^2)}=0,
 \label{eq:fft-kernel-l2}\\
 &\sup_{\substack{h\ge8,\ h\ \text{even}\\x,y\in\T^2}}
 \left\{K_h^{\T}(x,y)-
 \log\frac1{d_{\T}(x,y)\vee h^{-1}}\right\}<\infty,
 \label{eq:fft-log-bound}\\
 &\lim_{u\to\infty}
 \sup_{\substack{0\le\gamma\le\Gamma\\h\ge8,\ h\ \text{even}}}
 \EE\left[M_{\gamma,h}^{\T}\one\{M_{\gamma,h}^{\T}>u\}\right]=0,
 \qquad 0\le\Gamma<2.
 \label{eq:fft-ui}
\end{align}
\end{lemma}

\begin{proof}
Fix a frequency cutoff $L$.  For $h>2L$, the residues with
$0<|\widetilde r|\le L$ are canonically identified with integer frequencies,
\[
 \widehat K_h(r)\longrightarrow(2\pi|r|^2)^{-1},
 \qquad
 \phi_{h,r}\longrightarrow e^{2\pi i r\cdot x}
 \quad\text{in }L^2(\T^2).
\]
Orthogonality gives
\begin{equation}
 \sup_h\sum_{|\widetilde r|>L}\widehat K_h(r)^2
 \le\frac1{4\pi^2}\sum_{|k|>L}|k|^{-4}\longrightarrow0.
 \label{eq:fft-tail-l2}
\end{equation}
The continuous Fourier tail satisfies the same bound, so a low--high frequency
split proves \eqref{eq:fft-kernel-l2}.

Define the full heat sum
\begin{equation}
 U_h(z)=\sum_{k\in\Z^2\setminus\{0\}}
 \frac{e^{-9|k|^2/h^2}}{2\pi|k|^2}e^{2\pi i k\cdot z}.
 \label{eq:full-heat-sum}
\end{equation}
Grouping the full Fourier sum by residues modulo $h$ writes the difference
from the mixed half-open representative sum as an $h^{-2}$-weighted Fourier
sum of a uniformly bounded alias profile.  Hence
\begin{equation}
 \sup_{x,y\in\T^2}
 \left|K_h^{\T}(x,y)-U_h(\pi_hx-\pi_hy)\right|\le C.
 \label{eq:fft-heat-comparison}
\end{equation}
Poisson summation yields
\begin{equation}
 U_h(z)\le\log\frac1{d_{\T}(z,0)\vee h^{-1}}+C_0.
 \label{eq:heat-log-bound}
\end{equation}
Since $\pi_h$ moves each point by at most $\sqrt2/h$, the truncated logarithms
at $\pi_hx-\pi_hy$ and $x-y$ differ by at most an absolute constant.
Equations \eqref{eq:fft-heat-comparison}--\eqref{eq:heat-log-bound} prove
\eqref{eq:fft-log-bound}.

For \eqref{eq:fft-ui}, cover $\T^2$ by finitely many open sets $O$ of diameter
less than $1/2$.  By \eqref{eq:fft-log-bound}, in each Euclidean chart,
\begin{equation}
 K_h^{\T}(x,y)\le\log\frac1{|x-y|\vee h^{-1}}+C_1,
 \qquad x,y\in O.
 \label{eq:local-log-upper-bound}
\end{equation}
Choose a fixed $c_0>1$ sufficiently large and set
\begin{equation}
 K_{h,\Gamma}^{\rm ref}(x,y)=
 \Gamma^2\int_1^{c_0h}e^{-u|x-y|}\frac{\dd u}{u}.
 \label{eq:reference-covariance}
\end{equation}
Then, uniformly on $O^2$,
\begin{equation}
 \gamma^2K_h^{\T}(x,y)
 \le\gamma^2\left\{\log\frac1{|x-y|\vee h^{-1}}+C_1\right\}
 \le K_{h,\Gamma}^{\rm ref}(x,y)+C_3,
 \qquad 0\le\gamma\le\Gamma.
 \label{eq:legal-kahane-comparison}
\end{equation}
The reference masses associated with $K_{h,\Gamma}^{\rm ref}$ are uniformly
integrable for $\Gamma<2$ by \cite[Theorem~27]{Shamov}.  Adding the constant
covariance $C_3$ multiplies the reference mass by an independent mean-one
lognormal factor.  De la Vall\'ee--Poussin and Kahane's convexity inequality
\cite[Theorem~28]{Shamov} transfer uniform integrability to
$\mu_{\gamma,h}^{\T}(O)$.  Summing over the finite cover proves
\eqref{eq:fft-ui}.
\end{proof}

Fix $0\le\Gamma<2$, an even sequence $h_j\to\infty$, and
$\gamma_j\to\gamma\in[0,\Gamma]$.

\begin{lemma}
\label{lem:fft-chaos-convergence}
\begin{equation}
 \int_{\T^2}f\,\dd\mu_{\gamma_j,h_j}^{\T}
 \xrightarrow{\mathbb P}
 \int_{\T^2}f\,\dd\mu_\gamma^{\T}
 \qquad\bigl(f\in C(\T^2)\bigr).
 \label{eq:fft-chaos-convergence}
\end{equation}
\end{lemma}

\begin{proof}
Let $A_h$ and $A$ be the nonnegative square roots of the covariance operators
on $L^2(\T^2;\R)$ with kernels $K_h^{\T}(x,y)$ and $K_{\T}(x-y)$.
Equation~\eqref{eq:fft-kernel-l2} gives convergence of the covariance operators
in Hilbert--Schmidt, hence operator, norm.  The square-root inequality yields
\begin{equation}
 \norm{A_h-A}_{\mathrm{op}}
 \le\norm{A_h^2-A^2}_{\mathrm{op}}^{1/2}\longrightarrow0.
 \label{eq:square-root-convergence}
\end{equation}
For every $\xi\in L^2(\T^2;\R)$,
\begin{equation}
 \norm{\gamma_jA_{h_j}\xi-\gamma A\xi}_2
 \le \Gamma\norm{A_{h_j}-A}_{\mathrm{op}}\norm{\xi}_2
      +|\gamma_j-\gamma|\norm{A\xi}_2\longrightarrow0.
 \label{eq:generalized-field-convergence}
\end{equation}
Moreover,
\begin{align}
 &\norm{\gamma_j^2K_{h_j}^{\T}(x,y)-\gamma^2K_{\T}(x-y)}_2 \notag\\
 &\qquad\le
 \Gamma^2\norm{K_{h_j}^{\T}(x,y)-K_{\T}(x-y)}_2
 +|\gamma_j^2-\gamma^2|\norm{K_{\T}(x-y)}_2
 \longrightarrow0.
 \label{eq:scaled-kernel-convergence}
\end{align}
Together with \eqref{eq:fft-ui}, these are the hypotheses of
Proposition~\ref{prop:gmc-approximation}.

Choose a countable convergence-determining class
$(f_m)_{m\ge1}\subset C(\T^2)$ with $\norm{f_m}_\infty\le1$.  Proposition
\ref{prop:gmc-approximation} gives
\begin{equation}
 \sum_{m\ge1}2^{-m}
 \left(1\wedge\left|\int f_m\dd\mu_{\gamma_j,h_j}^{\T}
                  -\int f_m\dd\mu_\gamma^{\T}\right|\right)
 \longrightarrow0
 \quad\text{in probability}.
 \label{eq:measure-metric-convergence}
\end{equation}
GMC uniqueness identifies the limit with the standard zero-mode torus chaos,
and \eqref{eq:measure-metric-convergence} proves
\eqref{eq:fft-chaos-convergence}.
\end{proof}

\subsection{Quantitative periodic approximation}
\label{subsec:periodic-rate}

The quantitative argument has three inputs: a dyadic covariance increment,
uniform GMC mass bounds along the interpolation, and a truncation of the Cox
count.  Put
\begin{equation}
 d_h(x,y)=d_{\T}(x,y)\vee h^{-1}.
 \label{eq:periodic-cutoff-distance}
\end{equation}

\begin{lemma}
\label{lem:periodic-covariance-rate}
\begin{align}
 &\sup_{\substack{h\ge8,\ h\ \text{even}\\x,y\in\T^2}}
 \left|K_h^{\T}(x,y)-\log\frac1{d_h(x,y)}\right|\le C,
 \label{eq:periodic-two-sided-log}\\
 &\sup_{\substack{h\ge8,\ h\ \text{even}\\x,y\in\T^2}}
 hd_h(x,y)|K_{2h}^{\T}(x,y)-K_h^{\T}(x,y)|\le C,
 \label{eq:periodic-dyadic-covariance}\\
 &\sup_{h\ge8,\ h\ \text{even}}h^\beta
 \iint_{(\T^2)^2}|K_{2h}^{\T}(x,y)-K_h^{\T}(x,y)|
       d_h(x,y)^{-2+\beta}\dd x\dd y\le C_\beta,
 \qquad 0<\beta<1,
 \label{eq:periodic-weighted-covariance}\\
 &\sup_{\substack{h\ge8,\ h\ \text{even}\\0\le\gamma\le\Gamma}}
 \frac1{\omega_\Gamma(h)}
 \iint_{(\T^2)^2}
 \left|e^{\gamma^2K_h^{\T}(x,y)}
       -e^{\gamma^2K_{\T}(x-y)}\right|\dd x\dd y
 \le C_\Gamma,
 \qquad 0\le\Gamma<\sqrt2.
 \label{eq:periodic-exponential-rate}
\end{align}
\end{lemma}

\begin{proof}
Define
\begin{equation}
 \chi_0(z)=\frac{e^{-9|z|^2}}{2\pi|z|^2},
 \qquad
 \chi_1(z)=-\sum_{m\in\Z^2\setminus\{0\}}\chi_0(z+m).
 \label{eq:periodic-alias-functions}
\end{equation}
Grouping the full heat sum $U_h$ from \eqref{eq:full-heat-sum} by residues
modulo $h$ gives, at a grid difference $j/h$,
\begin{equation}
 K_h^{\T}(j/h,0)-U_h(j/h)
 =h^{-2}\sum_{r\in\Lambda_h}\chi_1(\widetilde r/h)
              e^{2\pi i\widetilde r\cdot j/h}.
 \label{eq:periodic-alias-identity}
\end{equation}
At the zero residue,
\[
 \chi_1(0)=-\sum_{m\ne0}\chi_0(m),
\]
so no value $\chi_0(0)$ is used.  Radial symmetry makes the values of $\chi_1$
agree on opposite faces of the mixed representative rectangle.  Extend
$\chi_1$ periodically and define, before using summation by parts,
\begin{equation}
 \Delta_{h,i}^2\chi_1(z)
 =\chi_1(z+h^{-1}e_i)-2\chi_1(z)+\chi_1(z-h^{-1}e_i),
 \qquad i=1,2.
 \label{eq:periodic-second-difference}
\end{equation}
The periodic extension is continuous and piecewise $C^2$, with only normal
first-derivative jumps at the seams; hence
\begin{equation}
 \sum_{r\in\Lambda_h}
 |\Delta_{h,i}^2\chi_1(\widetilde r/h)|\le C,
 \qquad i=1,2.
 \label{eq:periodic-alias-differences}
\end{equation}
Twice summing \eqref{eq:periodic-alias-identity} by parts gives
\begin{equation}
 |K_h^{\T}(j/h,0)-U_h(j/h)|
 \le \frac{C}{1+|j|_\infty^2}.
 \label{eq:periodic-alias-decay}
\end{equation}
The summation-by-parts multiplier has modulus two at a Nyquist spatial index,
so the estimate does not degenerate; each self-conjugate frequency residue
occurs once in the finite-group covariance.

Poisson summation gives
\begin{equation}
 U_h(z)=\frac1{2\pi}\int_{9/h^2}^{\infty}
 \left\{\frac{\pi}{t}\sum_{m\in\Z^2}
 e^{-\pi^2|z-m|^2/t}-1\right\}\dd t.
 \label{eq:periodic-heat-representation}
\end{equation}
Isolating the nearest lattice image yields
\begin{equation}
 \left|U_h(z)-\log\frac1{d_{\T}(z,0)\vee h^{-1}}\right|\le C,
 \qquad
 |\nabla U_h(z)|\le\frac{C}{d_{\T}(z,0)\vee h^{-1}}.
 \label{eq:periodic-heat-bounds}
\end{equation}
Changing only the lower heat cutoff in \eqref{eq:periodic-heat-representation}
also gives
\begin{equation}
 |U_{2h}(z)-U_h(z)|
 \le C\left(e^{-c h^2d_{\T}(z,0)^2}+h^{-2}\right)
 \le\frac{C}{1+\{h d_{\T}(z,0)\}^2}.
 \label{eq:periodic-heat-dyadic}
\end{equation}
Equations \eqref{eq:periodic-alias-decay}--\eqref{eq:periodic-heat-dyadic},
together with
\[
 d_{\T}(\pi_{2h}x-\pi_{2h}y,\pi_hx-\pi_hy)\le\frac{\sqrt2}{2h},
\]
prove \eqref{eq:periodic-two-sided-log} and
\eqref{eq:periodic-dyadic-covariance}.  Hence, by polar integration,
\begin{align}
 &\iint_{(\T^2)^2}|K_{2h}^{\T}-K_h^{\T}|(x,y)
       d_h(x,y)^{-2+\beta}\dd x\dd y \notag\\
 &\qquad\le Ch^{-\beta}
 +\frac Ch\int_{1/h}^{\diam(\T^2)}r^{-2+\beta}\dd r
 \le C_\beta h^{-\beta},
 \label{eq:periodic-weighted-proof}
\end{align}
which is \eqref{eq:periodic-weighted-covariance}.

For the last bound, the heat representation and the off-diagonal gradient
bound for the torus Green kernel imply, whenever $d_{\T}(x,y)\ge c/h$,
\begin{equation}
 |K_h^{\T}(x,y)-K_{\T}(x-y)|
 \le\frac{C}{h\,d_{\T}(x,y)}.
 \label{eq:periodic-target-pointwise}
\end{equation}
On the complementary disk use the logarithmic upper bounds for both kernels;
on its complement use the mean value theorem and
\eqref{eq:periodic-target-pointwise}.  With $a=\Gamma^2<2$, the radial bound is
\begin{equation}
 Ch^{a-2}+\frac Ch\int_{c/h}^{\diam(\T^2)}r^{-a}\dd r,
 \label{eq:periodic-exponential-integral}
\end{equation}
which is $O(h^{-1})$, $O(h^{-1}\log h)$, or $O(h^{-(2-a)})$ according as
$a<1$, $a=1$, or $1<a<2$.  This proves
\eqref{eq:periodic-exponential-rate}.
\end{proof}

For $a\in[0,a_+]$ and $t\in[0,1]$, let $\mu_{a,h,t}$ be the normalized
lognormal measure associated with covariance
\begin{equation}
 aK_{h,t},
 \qquad
 K_{h,t}=(1-t)K_h^{\T}+tK_{2h}^{\T},
 \label{eq:periodic-covariance-interpolation}
\end{equation}
and set
\begin{equation}
 M_{a,h,t}=\mu_{a,h,t}(\T^2),
 \qquad
 M_{a,h,t}^{x,y}=\int_{\T^2}
 e^{aK_{h,t}(x,z)+aK_{h,t}(y,z)}\mu_{a,h,t}(\dd z).
 \label{eq:periodic-inserted-mass}
\end{equation}
Assume
\begin{equation}
 0<a_+<4,
 \qquad
 1<p<4/a_+,
 \qquad
 0<s<1.
 \label{eq:periodic-moment-parameters}
\end{equation}

\begin{lemma}
\label{lem:periodic-palm-moments}
\begin{align}
 &\sup_{\substack{h\ge8,\ h\ \text{even}\\0\le t\le1\\0\le a\le a_+}}
 \EE M_{a,h,t}^{p}\le C_{a_+,p},
 \label{eq:periodic-positive-moment}\\
 &\sup_{\substack{h\ge8,\ h\ \text{even}\\0\le t\le1\\0\le a\le a_+\\x,y\in\T^2}}
 d_h(x,y)^{2-\beta(a_+,s)}
 e^{aK_{h,t}(x,y)}\EE[(M_{a,h,t}^{x,y})^{-s}]
 \le C_{a_+,s}.
 \label{eq:periodic-insertion-bound}
\end{align}
\end{lemma}

\begin{proof}
The GMC moment bounds in
\cite[Theorems~2.11--2.12]{RhodesVargas}, transferred uniformly to this
cutoff family by Kahane comparison, give \eqref{eq:periodic-positive-moment}
and the local negative-moment estimate below.  Put $r=d_h(x,y)$.  By
\eqref{eq:periodic-two-sided-log}, a geodesic square $B$ of side comparable to
$r$ can be chosen near $x$ so that
\begin{equation}
 M_{a,h,t}^{x,y}\ge c_{a_+}r^{-2a}\mu_{a,h,t}(B).
 \label{eq:inserted-mass-lower-bound}
\end{equation}
The same comparison gives
\begin{equation}
 \EE\mu_{a,h,t}(B)^{-s}
 \le Cr^{-2s-a(s+s^2)/2}.
 \label{eq:local-negative-moment}
\end{equation}
Combining \eqref{eq:inserted-mass-lower-bound},
\eqref{eq:local-negative-moment}, and the logarithmic upper bound for
$K_{h,t}$ gives
\begin{align}
 e^{aK_{h,t}(x,y)}\EE[(M_{a,h,t}^{x,y})^{-s}]
 &\le Cr^{-a+2as-2s-a(s+s^2)/2}\notag\\
 &=Cr^{-2+\beta(a,s)}
 \le Cr^{-2+\beta(a_+,s)},
 \label{eq:insertion-exponent-calculation}
\end{align}
where the last inequality uses that $a\mapsto\beta(a,s)$ is decreasing.
This is \eqref{eq:periodic-insertion-bound}.
\end{proof}

Assume in addition that $0<\beta(a_+,s)<1$, and fix an even $h\ge8$,
$0\le\gamma^2\le a_+$, and $n\ge1$.

\begin{lemma}
\label{lem:periodic-dyadic-law}
\begin{equation}
 \max\left\{
 d_{\TV}\bigl(\Law(\Xi_{\gamma,n,h}^{\T}),
               \Law(\Xi_{\gamma,n,2h}^{\T})\bigr),
 d_{\TV}(Q_{\gamma,n,h}^{\T},Q_{\gamma,n,2h}^{\T})
 \right\}
 \le1\wedge C_{a_+,p,s}
       \left(n^2h^{-\beta(a_+,s)}\right)^{p/(p+s)}.
 \label{eq:periodic-dyadic-law}
\end{equation}
\end{lemma}

\begin{proof}
It is enough to prove the Cox-process bound because a common connection and
observation kernel contracts total variation.  Put $a=\gamma^2$ and write
\begin{equation}
 \mu_{a,h,t}(\dd x)=W_t(x)\dd x,
 \qquad
 N\mid\mu_{a,h,t}\sim\operatorname{Pois}(nM_{a,h,t}).
 \label{eq:periodic-interpolated-count}
\end{equation}
Write $\PP_t,\EE_t$ for probability and expectation under this interpolated
Cox law.  Let $0\le H\le1$ be measurable on finite point configurations and set
$H_L(\xi)=H(\xi)\one\{|\xi|\le L\}$.  The Poisson moment inequality and
\eqref{eq:periodic-positive-moment} give
\begin{equation}
 \max_{t\in\{0,1\}}\PP_t(N>L)
 \le C\left(\frac{n}{L+1}\right)^p.
 \label{eq:periodic-count-tail}
\end{equation}
Both interpolating covariances are constant on the common partition into
resolution-$2h$ pixels.  Conditional on the cell counts, point locations are
independent uniform points in their cells.  Averaging $H_L$ over those
locations gives a bounded function of the count vector that is independent of
$t$, and
\begin{equation}
 |\nabla_{x,y}^+H_L|\le2\one\{N\le L\}.
 \label{eq:truncated-add-two-bound}
\end{equation}
Applying \eqref{eq:cox-cancellation} and rewriting the cell sums as integrals
yields
\begin{align}
 \left|\frac{\dd}{\dd t}\EE_tH_L\right|
 &\le C_{a_+} n^2\iint |K_{2h}^{\T}-K_h^{\T}|(x,y)\notag\\
 &\qquad\times
 \EE_t\{W_t(x)W_t(y)\one\{N\le L\}\}\dd x\dd y.
 \label{eq:periodic-truncated-interpolation}
\end{align}
For $Z_\lambda\sim\operatorname{Poisson}(\lambda)$ and $s>0$,
\begin{equation}
 \PP(Z_\lambda\le L)\le C_s(L+1)^s\lambda^{-s}.
 \label{eq:periodic-poisson-lower-tail}
\end{equation}
Indeed,
\begin{align}
 \PP(Z_\lambda\le L)
 &\le(L+1)^s\EE(Z_\lambda+1)^{-s},\notag\\
 \EE(Z_\lambda+1)^{-s}
 &=\frac1{\Gamma(s)}\int_0^\infty
 t^{s-1}e^{-t+\lambda(e^{-t}-1)}\dd t
 \le C_s\lambda^{-s}.
 \label{eq:poisson-negative-moment}
\end{align}
For $\lambda\ge1$, split the integral at $t=1$ and use
$1-e^{-t}\ge(1-e^{-1})t$ on $0\le t\le1$; for $0<\lambda<1$, the last bound
is immediate.

The two-point Palm formula for log-Gaussian Cox processes
\cite{CoeurjollyMollerWaagepetersen} gives
\begin{equation}
 \EE_t\{W_t(x)W_t(y)\one\{N\le L\}\}
 =e^{aK_{h,t}(x,y)}
 \EE_t\!\left[
 \PP\{Z_{nM_{a,h,t}^{x,y}}\le L\mid\mu_{a,h,t}\}\right].
 \label{eq:periodic-double-shift}
\end{equation}
Equations \eqref{eq:periodic-poisson-lower-tail},
\eqref{eq:periodic-insertion-bound}, and
\eqref{eq:periodic-weighted-covariance} imply
\begin{equation}
 |\EE_1H_L-\EE_0H_L|
 \le Cn^{2-s}(L+1)^sh^{-\beta(a_+,s)}.
 \label{eq:periodic-truncated-bound}
\end{equation}
At the two endpoints,
\begin{equation}
 |\EE_1H-\EE_0H|
 \le \PP_1(N>L)+\PP_0(N>L)+|\EE_1H_L-\EE_0H_L|.
 \label{eq:periodic-truncation-decomposition}
\end{equation}
Combining \eqref{eq:periodic-count-tail},
\eqref{eq:periodic-truncated-bound}, and
\eqref{eq:periodic-truncation-decomposition}, choose
\begin{equation}
 L+1\asymp n\{n^2h^{-\beta(a_+,s)}\}^{-1/(p+s)}
 \label{eq:periodic-optimal-truncation}
\end{equation}
when $n^2h^{-\beta(a_+,s)}\le1$.  Then
\begin{equation}
 \PP_1(N>L)+\PP_0(N>L)+|\EE_1H_L-\EE_0H_L|
 \le C\left(n^2h^{-\beta(a_+,s)}\right)^{p/(p+s)}.
 \label{eq:periodic-balanced-errors}
\end{equation}
Integer rounding changes only the constant; when
$n^2h^{-\beta(a_+,s)}>1$, the bound by one is immediate.  Taking the supremum
over $H$ proves \eqref{eq:periodic-dyadic-law}.
\end{proof}

\subsection{Square-domain cutoff and pixelization bounds}
\label{subsec:explicit-error}

We next specialize Lemma~\ref{lem:continuum-tv-bound} to an explicit
square-domain cutoff and its pixelization.  Set
\begin{equation}
 K_*(x,y)=\int_1^\infty e^{-u^2|x-y|^2}\frac{\dd u}{u},
 \qquad
 K_{*,\varepsilon}(x,y)=\int_1^{1/\varepsilon}
                 e^{-u^2|x-y|^2}\frac{\dd u}{u}.
 \label{eq:star-covariance}
\end{equation}
Let $P_{\gamma,n}^{*}$ denote the continuous graph law generated by $K_*$.  This is
a scale-integral covariance of the type used in \cite{FGN}.  Let
$Q_{\gamma,n,\varepsilon}$ be the graph law directed by the chaos of
$K_{*,\varepsilon}$, and let $Q_{\gamma,n,\varepsilon,\delta}$ be the graph law
obtained by evaluating that field at centers of pixels of side $\delta$ and
holding it constant inside each pixel.  Define
\begin{equation}
 E_1(z)=\int_z^\infty e^{-t}\frac{\dd t}{t}.
 \label{eq:exp-integral-definition}
\end{equation}
Then
\begin{equation}
 K_*(x,y)=\tfrac12E_1(|x-y|^2),
 \qquad
 K_*(x,y)-K_{*,\varepsilon}(x,y)
 =\tfrac12E_1(|x-y|^2/\varepsilon^2).
 \label{eq:exp-integral}
\end{equation}
On the unit square, put
\begin{equation}
 \sup_{0<r\le\sqrt2}
 \left|\tfrac12E_1(r^2)+\log r\right|<0.372.
 \label{eq:exp-integral-constant}
\end{equation}
Thus
\[
 K_*(x,y)\le\log\frac1{|x-y|}+0.372.
\]
With $\gamma_{\rm E}$ denoting Euler's constant,
\begin{equation}
 \frac{\dd}{\dd r}
 \left\{\tfrac12E_1(r^2)+\log r\right\}
 =\frac{1-e^{-r^2}}{r}>0,
 \qquad
 \lim_{r\downarrow0}
 \left\{\tfrac12E_1(r^2)+\log r\right\}
 =-\frac{\gamma_{\rm E}}2,
 \label{eq:exp-integral-monotonicity}
\end{equation}
and
\begin{equation}
 \tfrac12E_1(2)+\tfrac12\log2<0.372.
 \label{eq:exp-integral-endpoint}
\end{equation}

Fix $0<\Gamma<\sqrt2$ and define
\begin{equation}
 C_{\rm cut}(\Gamma)=2\pi e^{0.372\Gamma^2}
       \left(\frac1{2-\Gamma^2}+\frac{\Gamma^2}{4e}\right).
 \label{eq:cutoff-constant}
\end{equation}
For
\begin{equation}
 0\le\gamma\le\Gamma,
 \qquad
 n\le n_+,
 \qquad
 0<\varepsilon<1,
 \qquad
 \delta>0,
 \label{eq:explicit-cutoff-parameters}
\end{equation}
the following bounds hold.

\begin{corollary}
\label{cor:explicit-cutoff}
\begin{align}
 d_{\TV}(P_{\gamma,n}^{*},Q_{\gamma,n,\varepsilon})
 &\le1\wedge n_+^2C_{\rm cut}(\Gamma)\varepsilon^{2-\Gamma^2},
 \label{eq:cutoff-error}\\
 d_{\TV}(Q_{\gamma,n,\varepsilon},Q_{\gamma,n,\varepsilon,\delta})
 &\le1\wedge\frac{2\Gamma^2}{\sqrt e}n_+^2
                   \delta\varepsilon^{-1-\Gamma^2}.
 \label{eq:pixel-error}
\end{align}
\end{corollary}

\begin{proof}
The nonnegative covariances in \eqref{eq:star-covariance} are ordered, so the
exponential difference is increasing in $\gamma^2$.  With $z=x-y$ and
$D-D\subset\{|z|\le\sqrt2\}$,
\begin{align}
 &\iint_{D^2}(e^{\Gamma^2K_*(x,y)}-e^{\Gamma^2K_{*,\varepsilon}(x,y)})\dd x\dd y\notag\\
 &\qquad\le2\pi e^{0.372\Gamma^2}\varepsilon^{2-\Gamma^2}
   \int_0^\infty u^{1-\Gamma^2}
      \bigl(1-e^{-(\Gamma^2/2)E_1(u^2)}\bigr)\dd u.
 \label{eq:radial-bound}
\end{align}
On $0<u\le1$ use
$1-e^{-(\Gamma^2/2)E_1(u^2)}\le1$.  For $u\ge1$,
\[
 1-e^{-(\Gamma^2/2)E_1(u^2)}
 \le\frac{\Gamma^2}{2}\frac{e^{-u^2}}{u^2},
 \qquad
 u^{-1-\Gamma^2}\le u.
\]
Thus
\begin{equation}
 \int_0^1u^{1-\Gamma^2}\dd u=\frac1{2-\Gamma^2},
 \qquad
 \frac{\Gamma^2}{2}\int_1^\infty u^{-1-\Gamma^2}e^{-u^2}\dd u
 \le\frac{\Gamma^2}{4e}.
 \label{eq:cutoff-radial-integrals}
\end{equation}
Lemma~\ref{lem:continuum-tv-bound} gives \eqref{eq:cutoff-error}.

Let $z_\delta(x)$ be the center of the pixel containing $x$.  Direct
differentiation in the difference variable yields
\begin{equation}
 \sup_z|\nabla K_{*,\varepsilon}(z)|
 \le\sqrt{2/e}\,\varepsilon^{-1},
 \qquad
 |(x-y)-\{z_\delta(x)-z_\delta(y)\}|\le\sqrt2\,\delta.
 \label{eq:pixel-covariance-lipschitz}
\end{equation}
Both covariances lie in $[0,\log(1/\varepsilon)]$, so the mean value theorem
gives
\begin{equation}
 \left|e^{\gamma^2K_{*,\varepsilon}(x,y)}
 -e^{\gamma^2K_{*,\varepsilon}(z_\delta(x),z_\delta(y))}\right|
 \le\frac{2\Gamma^2}{\sqrt e}
       \delta\varepsilon^{-1-\Gamma^2}.
 \label{eq:pixel-exponential-bound}
\end{equation}
Lemma~\ref{lem:regular-tv} proves \eqref{eq:pixel-error}.
\end{proof}

The same approximation budget applies to the independently piloted residual
on its fixed partition.  Assume deterministic numbers $\eta_B$ satisfy
\begin{equation}
 \eta_B\ge\sup_{\theta\in B}d_{\TV}(P_\theta,Q_{\theta,h}),
 \qquad B\in\mathcal B,
 \label{eq:residual-box-error}
\end{equation}
and define
\begin{equation}
 \widetilde p_B^\rho=\min\{1,p_B^\rho+\eta_B\},
 \qquad
 \widetilde C_\alpha^\rho(G)=
  \bigcup_{B\in\mathcal B:\,\widetilde p_B^\rho(G)>\alpha}
       \{\gamma^2/2:(\gamma,n)\in B\}.
 \label{eq:corrected-residual-set}
\end{equation}
Fix $\theta=(\gamma,n)\in\Theta$.  For almost every pilot realization
$\mathcal D_0$, the following assertion holds.

\begin{corollary}
\label{cor:residual-transfer}
\begin{equation}
 (P_\theta\otimes P_h^{\rm aux})
 \{\gamma^2/2\in\widetilde C_\alpha^\rho(G)\mid\mathcal D_0\}
 \ge1-\alpha.
 \label{eq:residual-transfer}
\end{equation}
\end{corollary}

\begin{proof}
Condition on $\mathcal D_0$ and choose a fixed $B_*\ni\theta$.  If
$\eta_{B_*}\ge\alpha$, then
\begin{equation}
 \widetilde p_{B_*}^\rho
 \ge\min\left\{1,\alpha+\frac1{R+1}\right\}>\alpha.
 \label{eq:residual-large-error-retention}
\end{equation}
If $\eta_{B_*}<\alpha$, then
\begin{equation}
 \{\widetilde p_{B_*}^\rho\le\alpha\}
 =\{p_{B_*}^\rho\le\alpha-\eta_{B_*}\},
 \label{eq:residual-shifted-rejection}
\end{equation}
and \eqref{eq:main-residual-size} with
$u=\alpha-\eta_{B_*}$ gives
\begin{equation}
 (Q_{\theta,h}\otimes P_h^{\rm aux})
 \{\widetilde p_{B_*}^\rho\le\alpha\mid\mathcal D_0\}
 \le\alpha-\eta_{B_*}.
 \label{eq:residual-simulator-size}
\end{equation}
Since the calibration sample has the same law under both models,
\begin{align}
 &(P_\theta\otimes P_h^{\rm aux})
 \{\widetilde p_{B_*}^\rho\le\alpha\mid\mathcal D_0\}\notag\\
 &\qquad\le
 (Q_{\theta,h}\otimes P_h^{\rm aux})
 \{\widetilde p_{B_*}^\rho\le\alpha\mid\mathcal D_0\}
 +d_{\TV}(P_\theta,Q_{\theta,h})
 \le\alpha.
 \label{eq:residual-target-size}
\end{align}
Noncoverage implies $\widetilde p_{B_*}^\rho\le\alpha$, which proves
\eqref{eq:residual-transfer}.
\end{proof}

\subsection{Final proof of Theorem~\ref{thm:transfer}}
\label{subsec:final-transfer}

The assumed weak convergence of the chaos measures and
Lemma~\ref{lem:graph-tv-convergence} give \eqref{eq:main-qualitative}, while
Lemma~\ref{lem:continuum-tv-bound} gives \eqref{eq:main-tv}.  Taking the
pointwise supremum over $\gamma$ in \eqref{eq:continuum-tv-bound} and the upper
endpoint of $n$ gives \eqref{eq:eta-box}.

Fix the true parameter $\theta$ and set
\begin{equation}
 \eta_\theta=d_{\TV}(P_\theta,Q_{\theta,h}).
 \label{eq:true-parameter-tv}
\end{equation}
For every rectangle $B\ni\theta$,
\begin{equation}
 p_B+\eta_B\ge q_\theta+\eta_\theta.
 \label{eq:corrected-rank-domination}
\end{equation}
Hence
\begin{equation}
 \{\exists j\ge0:\theta\notin\widetilde\Theta_j\}
 \subseteq\{q_\theta\le\alpha-\eta_\theta\},
 \label{eq:corrected-noncoverage-inclusion}
\end{equation}
where the right side is empty if $\eta_\theta\ge\alpha$.
Equation~\eqref{eq:true-rank-validity} therefore gives
\begin{equation}
 (Q_{\theta,h}\otimes P_h^{\rm aux})
 \{\exists j\ge0:\theta\notin\widetilde\Theta_j\}
 \le(\alpha-\eta_\theta)_+.
 \label{eq:corrected-simulator-bound}
\end{equation}
The auxiliary randomness has the same law under the target and simulator
models, so
\begin{equation}
 d_{\TV}(P_\theta\otimes P_h^{\rm aux},
          Q_{\theta,h}\otimes P_h^{\rm aux})=\eta_\theta.
 \label{eq:product-tv}
\end{equation}
Combining \eqref{eq:corrected-simulator-bound} and \eqref{eq:product-tv}
proves \eqref{eq:main-corrected}.

\subsection{Final proof of Theorem~\ref{thm:periodic-rate}}
\label{subsec:final-periodic-rate}

Fix $\gamma\le\Gamma$ and $n\le\bar n$.  Lemma~\ref{lem:periodic-dyadic-law}
at $h,2h,4h,\ldots$ gives, for every $m\ge1$,
\begin{align}
 &d_{\TV}\bigl(\Law(\Xi_{\gamma,n,h}^{\T}),
                 \Law(\Xi_{\gamma,n,2^mh}^{\T})\bigr)\notag\\
 &\qquad\le C\sum_{j=0}^{m-1}
 \left(\bar n^2(2^jh)^{-\beta_\Gamma(s)}\right)^{p/(p+s)}\notag\\
 &\qquad\le C_{\Gamma,p,s}
 \left(\bar n^2h^{-\beta_\Gamma(s)}\right)^{p/(p+s)}.
 \label{eq:periodic-dyadic-telescope}
\end{align}
For fixed $\gamma,n$, Lemmas~\ref{lem:fft-chaos-convergence} and
\ref{lem:graph-tv-convergence} give
\[
 \Xi_{\gamma,n,2^mh}^{\T}\Longrightarrow\Xi_{\gamma,n}^{\T}.
\]
Lower semicontinuity of total variation under weak convergence and
\eqref{eq:periodic-dyadic-telescope} yield
\begin{align}
 &d_{\TV}\bigl(\Law(\Xi_{\gamma,n,h}^{\T}),
                   \Law(\Xi_{\gamma,n}^{\T})\bigr)\notag\\
 &\qquad\le\liminf_{m\to\infty}
 d_{\TV}\bigl(\Law(\Xi_{\gamma,n,h}^{\T}),
                   \Law(\Xi_{\gamma,n,2^mh}^{\T})\bigr)\notag\\
 &\qquad\le C_{\Gamma,p,s}
 \left(\bar n^2h^{-\beta_\Gamma(s)}\right)^{p/(p+s)}.
 \label{eq:periodic-limit-tv-step}
\end{align}
This is the Cox-process bound in \eqref{eq:periodic-full-rate}; contraction
under the common connection-and-erasure kernel gives the graph bound.  The
constants are uniform in $\gamma\le\Gamma$ and $n\le\bar n$.

If $\Gamma<\sqrt2$, the covariance-interpolation argument of
Lemma~\ref{lem:continuum-tv-bound}, applied on $\T^2$ together with
\eqref{eq:periodic-exponential-rate}, gives
\begin{equation}
 d_{\TV}\bigl(\Law(\Xi_{\gamma,n,h}^{\T}),
               \Law(\Xi_{\gamma,n}^{\T})\bigr)
 \le1\wedge C_\Gamma n^2\omega_\Gamma(h).
 \label{eq:periodic-cox-second-moment}
\end{equation}
Total-variation contraction gives the graph bound
\eqref{eq:periodic-second-moment-rate}.  Setting $\bar n=n_+$ in
\eqref{eq:periodic-full-rate} proves \eqref{eq:main-periodic-tv}.  At
$\Gamma=0$ both directing measures equal Lebesgue measure.  Finally,
substitution of \eqref{eq:hard-range-choice} into \eqref{eq:rate-exponents}
gives \eqref{eq:hard-range-exponents}.

\begin{proof}[Proof of Corollary~\ref{cor:fft-continuum-coverage}]
Theorem~\ref{thm:calibration} and independence of the auxiliary randomness give
\begin{align}
 (P_\theta^{\T}\otimes P_h^{\rm aux})(F_{\theta,h})
 &\le (Q_{\theta,h}^{\T}\otimes P_h^{\rm aux})(F_{\theta,h})
      +d_{\TV}(P_\theta^{\T},Q_{\theta,h}^{\T})\notag\\
 &\le\alpha+\Delta_h(\Theta).
 \label{eq:periodic-coverage-transfer-step}
\end{align}
Taking the supremum over $\theta$ and using \eqref{eq:main-periodic-tv} proves
\eqref{eq:periodic-coverage-error}; taking complements and then the lower limit
through even $h$ gives \eqref{eq:periodic-asymptotic-coverage}.  The same
calculation holds conditionally on any fixed independent pilot outcome for
which the residual score is defined.
\end{proof}

\Needspace{10\baselineskip}
\subsection{The order of the limits}
\label{subsec:order-limits}

We finally compare the intensity limit with the resolution limit.  Keep the
finite-pixel covariance fixed while $n\to\infty$.  For each field realization
the sampled density is then bounded, so the degree statistic has a different
limit from the continuous GMC model.  Fix a finite-pixel model and
$0\le\gamma<2$.

\begin{proposition}
\label{prop:fixed-grid-limit}
\begin{equation}
 \forall\epsilon>0:\qquad
 \max\left\{
 Q_{\gamma,n,h}\{T_D>\epsilon\},
 Q_{\gamma,n,h}^{\T}\{T_D>\epsilon\}
 \right\}\longrightarrow0
 \qquad(n\to\infty),
 \label{eq:fixed-grid-degree}
\end{equation}
and
\begin{equation}
 0<\gamma<2
 \quad\Longrightarrow\quad
 d_{\TV}(P_{\gamma,n},Q_{\gamma,n,h})\longrightarrow1
 \qquad(n\to\infty).
 \label{eq:fixed-grid-separation}
\end{equation}
\end{proposition}

\begin{proof}
Condition on the finite Gaussian vector and write
\begin{equation}
 w=w_{\gamma,h},
 \qquad
 M=\int w\dd x,
 \qquad
 0<M<\infty,
 \qquad
 \norm w_\infty<\infty
 \quad\text{a.s.}
 \label{eq:fixed-grid-density-bounds}
\end{equation}
In either the square or periodic model, the Palm degree at $x$ is Poisson with
mean $\lambda_n(x)$ satisfying
\begin{equation}
 \lambda_n(x)\le(1\vee\pi c^2)\norm w_\infty.
 \label{eq:fixed-grid-palm-bound}
\end{equation}
Indeed, the square-model integral is bounded by $\pi c^2\norm w_\infty/n$,
while \eqref{eq:torus-connection} is bounded by integrating over a fundamental
square.  For $\epsilon>0$, put
\begin{equation}
 Z_n^{\rm grid}(\epsilon)=\#\{v:\log(1+d_v)>\epsilon\log n\}.
 \label{eq:fixed-grid-bad-count-definition}
\end{equation}
Mecke's identity and Markov's inequality give, in either model,
\begin{equation}
 \frac1n\EE[Z_n^{\rm grid}(\epsilon)\mid w]
 \le M\{1+(1\vee\pi c^2)\norm w_\infty\}n^{-\epsilon}\longrightarrow0.
 \label{eq:fixed-grid-bad-count}
\end{equation}
Since $N_n/n\to M>0$ conditionally in probability,
\begin{equation}
 \frac{Z_n^{\rm grid}(\epsilon)}{N_n}\longrightarrow0,
 \qquad
 \left\{\frac{\log(1+\med_vd_v)}{\log n}>\epsilon\right\}
 \subseteq\left\{Z_n^{\rm grid}(\epsilon)\ge\frac{N_n}{2}\right\},
 \label{eq:fixed-grid-median-inclusion}
\end{equation}
with $Z_n^{\rm grid}(\epsilon)/N_n:=0$ on $\{N_n=0\}$.  Therefore
\begin{equation}
 \frac{\log(1+\med_vd_v)}{\log n}\longrightarrow0,
 \qquad
 \frac{\log(N_n\vee2)}{\log n}\longrightarrow1
 \quad\text{conditionally in probability}.
 \label{eq:fixed-grid-statistic-limits}
\end{equation}
Hence $T_D\to0$ conditionally in probability, and bounded convergence gives
\eqref{eq:fixed-grid-degree}.

For $\gamma>0$, Theorem~\ref{thm:single} and
\eqref{eq:fixed-grid-degree} give
\begin{equation}
 P_{\gamma,n}\{T_D>\nu/2\}\longrightarrow1,
 \qquad
 Q_{\gamma,n,h}\{T_D>\nu/2\}\longrightarrow0.
 \label{eq:fixed-grid-separating-event}
\end{equation}
Hence
\begin{equation}
 1\ge d_{\TV}(P_{\gamma,n},Q_{\gamma,n,h})
 \ge P_{\gamma,n}\{T_D>\nu/2\}
    -Q_{\gamma,n,h}\{T_D>\nu/2\}\longrightarrow1,
\end{equation}
which proves \eqref{eq:fixed-grid-separation}.
\end{proof}

For a square-domain finite-pixel family satisfying the approximation premise
of Theorem~\ref{thm:transfer},
\begin{equation}
 \lim_{n\to\infty}\lim_{h\to\infty}
       d_{\TV}(P_{\gamma,n},Q_{\gamma,n,h})=0,
 \qquad
 \lim_{h\to\infty}\lim_{n\to\infty}
       d_{\TV}(P_{\gamma,n},Q_{\gamma,n,h})=1,
 \qquad 0<\gamma<2.
 \label{eq:noncommuting-limits}
\end{equation}
For every sequence $h_j\uparrow\infty$, \eqref{eq:fixed-grid-separation}
also permits $n_j\uparrow\infty$ such that
\[
 d_{\TV}(P_{\gamma,n_j},Q_{\gamma,n_j,h_j})>1-j^{-1}.
\]

The periodic estimate gives explicit sufficient joint regimes.  Fix
$0<\Gamma<2$ and admissible $p,s$ from
Theorem~\ref{thm:periodic-rate}.  If $\bar n_h\ge1$ satisfies
\begin{equation}
 \bar n_h=o\!\left(h^{\beta_\Gamma(s)/2}\right),
 \label{eq:diagonal-resolution}
\end{equation}
then
\begin{equation}
 \sup_{\substack{0\le\gamma\le\Gamma\\1\le n\le\bar n_h}}
 d_{\TV}(Q_{\gamma,n,h}^{\T},P_{\gamma,n}^{\T})\longrightarrow0.
 \label{eq:joint-periodic-limit}
\end{equation}
For $0<\Gamma<\sqrt2$, \eqref{eq:periodic-second-moment-rate} also allows
\begin{equation}
 \bar n_h=
 \begin{cases}
  o(h^{1/2}),&0<\Gamma<1,\\
  o\{(h/\log h)^{1/2}\},&\Gamma=1,\\
  o(h^{1-\Gamma^2/2}),&1<\Gamma<\sqrt2.
 \end{cases}
 \label{eq:second-moment-joint-regimes}
\end{equation}
For $\sqrt2\le\Gamma<2$, the choice
\eqref{eq:hard-range-choice} turns \eqref{eq:diagonal-resolution} into
$\bar n_h=o(h^{\beta_\Gamma(s_\Gamma)/2})$.  At $\Gamma=0$ the periodic finite and
continuous laws coincide for every $h$ and $n$.

\section{Numerical calibration and statistical performance}
\label{sec:numerics}

Three studies address distinct questions: whether parameter refinement improves
fixed-statistic inference, whether the size--degree residual improves estimation,
and how finite-resolution diagnostics change with $h$.  The first two use the
$h=256$ periodic finite model $Q_{\cdot,256}^{\T}$; the third compares several
spatial resolutions.

\subsection{Finite-model refinement}
\label{subsec:numerical-power}

The first study keeps the calibration library fixed and asks how much finer
parameter boxes improve testing and confidence-set resolution.  The simulator
retains the random total mass, samples vertices uniformly within pixels, and
keeps isolated vertices.  At $h=256$,
\[
 v_{256}=4.559397874424695,
\]
and the parameter domain is
\begin{equation}
 \Theta=[0,1.5]\times[1500,2500],\qquad 0\le\nu\le1.125.
 \label{eq:numerical-domain}
\end{equation}
Validation intensities are drawn uniformly from $[1500,2500]$ and are not
revealed to the inference procedure.  The replication unit is an independent
environment with its Poisson and edge randomness.

The first study uses $R=399$ independent mother simulations on the three
partitions in Table~\ref{tab:design}.  Validation uses 1200 independent graphs,
200 at each
$\gamma\in\{0.35,0.60,0.65,0.95,1.00,1.25\}$.

\begin{table}[htbp]
 \centering
 \caption{Parameter partitions.  The fixed-statistic study uses all levels;
 the residual study uses the fine partition.  The intensity interval is
 $[1500,2500]$.}
 \label{tab:design}
 \begin{tabular}{lrrr}
  \toprule
  Level & Width in $\gamma$ & Width in $n$ & Rectangles\\
  \midrule
  Coarse & 0.100 & 500 & 30\\
  Intermediate & 0.050 & 250 & 120\\
  Fine & 0.025 & 125 & 480\\
  \bottomrule
 \end{tabular}
\end{table}

We test
\begin{equation}
 H_0:\gamma\le0.6\quad\Longleftrightarrow\quad\nu\le0.18
 \label{eq:numerical-null}
\end{equation}
by maximizing the upper-tail value over rectangles covering the null region
and intensity range, as in \eqref{eq:composite-pvalue}.

\begin{table}[htbp]
 \centering
 \small
 \caption{Rejection percentages for \eqref{eq:numerical-null}, based on 200
 independent validation graphs per row.  Degree and edge columns are separate
 level-$0.05$ tests.}
 \label{tab:power}
 \begin{tabular}{rr rrr r}
  \toprule
  &&\multicolumn{3}{c}{Degree statistic $T_D$}&Edge $T_E$\\
  \cmidrule(lr){3-5}
  $\gamma$&$\nu$&Coarse&Intermediate&Fine&Fine\\
  \midrule
  0.35 & 0.06125 & 0.0 & 0.0 & 0.0 & 0.0 \\
  0.60 & 0.18000 & 0.0 & 0.0 & 1.0 & 2.0 \\
  0.65 & 0.21125 & 0.5 & 1.0 & 6.0 & 7.5 \\
  0.95 & 0.45125 & 35.5 & 56.0 & 81.5 & 82.5 \\
  1.00 & 0.50000 & 44.5 & 65.5 & 88.0 & 85.5 \\
  1.25 & 0.78125 & 83.0 & 90.5 & 98.5 & 96.5 \\
  \bottomrule
 \end{tabular}
\end{table}

\begin{figure}[!htbp]
 \centering
 \includegraphics[width=0.91\textwidth]{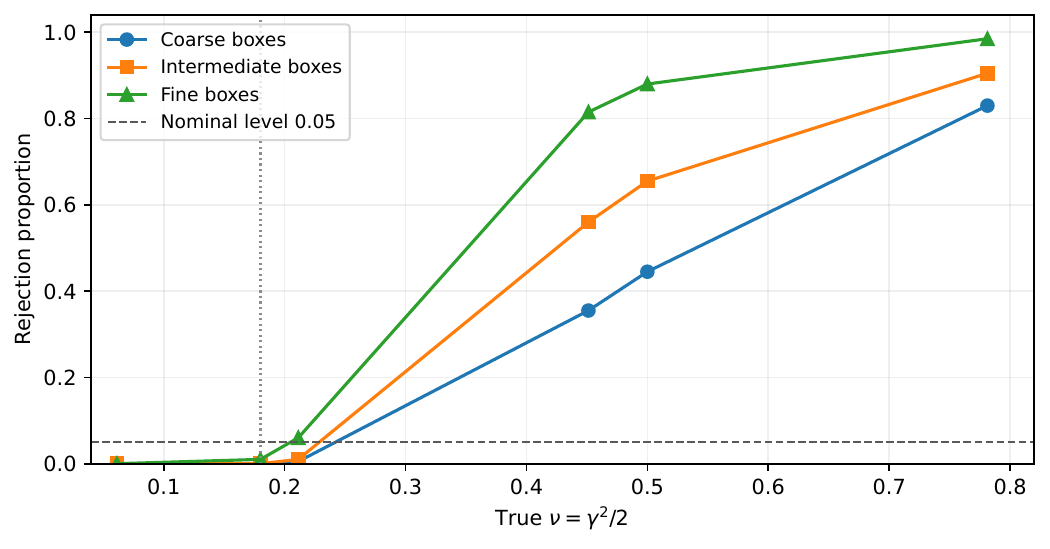}
 \caption{Degree-test rejection proportions.  The vertical line marks
 $\nu=0.18$ and the horizontal line the nominal level $0.05$.}
 \label{fig:power-refinement}
\end{figure}

Refinement raises the degree-test rejection frequency at $\nu=0.5$ from
$44.5\%$ to $88.0\%$ and at $\nu=0.78125$ from $83.0\%$ to $98.5\%$; at the
null boundary the fine degree test rejects $2/200$ graphs.  At
$\nu=0.21125$, the median degree confidence-set hull contracts from $1.125$ to
$0.4303$; the fine-level median lengths are $0.8972$ and $0.8050$ at
$\nu=0.5$ and $0.78125$.

Branch-and-prune returns the same fine-level confidence sets as exhaustive
evaluation on all 2400 statistic--graph pairs, with mean search counts of
$361.55$ rectangles for degree and $363.57$ for edge, compared with $630$ in
the full tree.

\FloatBarrier
\subsection{Size--degree residual calibration}
\label{subsec:numerical-residual-design}

The second study compares degree, edge, and the size--degree residual on the
same fine partition, asking whether the observed size improves point estimation
and confidence-set precision.  The residual score in \eqref{eq:residual-score}
was fixed before this study.  It uses 99 pilot environments, a new
399-environment calibration
library, and 1200 new validation graphs, 200 at each
\begin{equation}
 \gamma\in\{0.60,0.95,1.00,1.15,1.25,1.35\}.
 \label{eq:residual-test-parameters}
\end{equation}
The pilot, calibration, and validation samples are disjoint.  All methods use
the same 480 fine rectangles; for \eqref{eq:pilot-regularization},
\begin{equation}
 R_0=99,
 \qquad \lambda=0.05,
 \qquad (e_1,e_2)=(10^{-4},2.5\times10^{-5}),
 \qquad \varepsilon=10^{-10}.
 \label{eq:pilot-numerical-constants}
\end{equation}
Degree and edge use the two-sided box values \eqref{eq:box-twosided}; the
residual uses \eqref{eq:residual-calibration}.  All methods invert at level
$0.05$ and use \eqref{eq:calibrated-point-estimate}.

For each parameter value, define
\begin{align}
 \widehat{\mathrm{RMSE}}
 &=\left\{\frac1{200}\sum_{i=1}^{200}
       (\widehat\nu_{\mathrm{cal}}(G_i)-\nu)^2\right\}^{1/2},
 \label{eq:empirical-rmse}\\
 \shortintertext{and}
 \widehat{\mathrm{coverage}}
 &=\frac1{200}\sum_{i=1}^{200}\one\{\nu\in C_i^\nu\}.
 \label{eq:empirical-coverage}
\end{align}
Here $C_i^\nu$ is the projected accepted set; widths below are lengths of its
smallest closed interval hull.

\begin{table}[htbp]
 \centering
 \caption{Empirical RMSE in the residual study.  Each row contains 200
 independent validation graphs.}
 \label{tab:residual-rmse}
 \begin{tabular}{rrrr}
  \toprule
  $\nu$&Degree&Edge&Size--degree residual\\
  \midrule
  0.18000 & 0.0637 & 0.0615 & 0.0933 \\
  0.45125 & 0.1513 & 0.1816 & 0.0930 \\
  0.50000 & 0.1710 & 0.2000 & 0.0974 \\
  0.66125 & 0.2189 & 0.2433 & 0.1121 \\
  0.78125 & 0.2344 & 0.2467 & 0.1180 \\
  0.91125 & 0.2270 & 0.2518 & 0.1160 \\
  \bottomrule
 \end{tabular}
\end{table}

\begin{figure}[!htbp]
 \centering
 \includegraphics[width=0.91\textwidth]{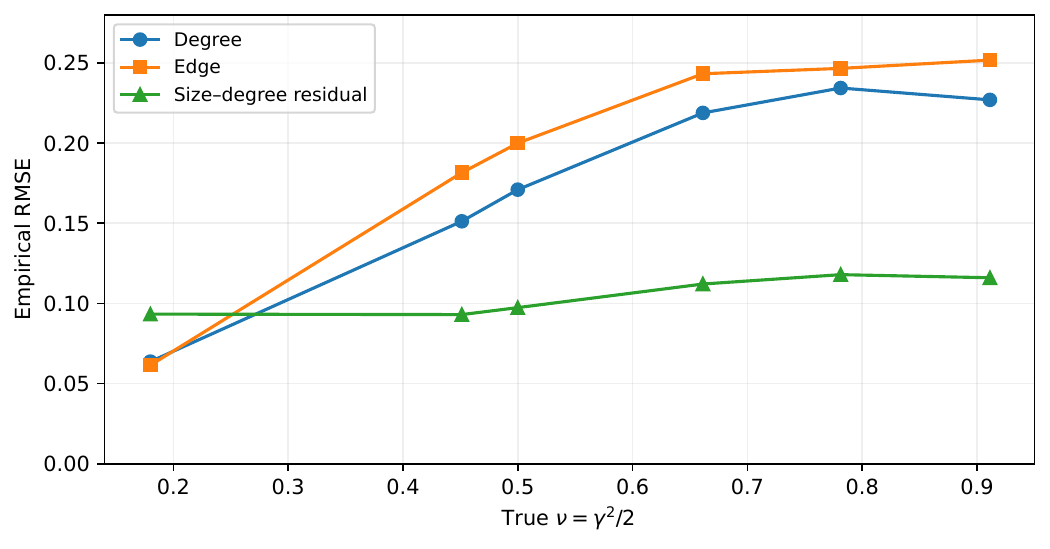}
 \caption{Empirical RMSE with the fine partition fixed; lower is better.}
 \label{fig:residual-rmse}
\end{figure}

For $\nu\ge0.45125$ the residual has smaller RMSE than the degree baseline; at
$\nu=0.78125$ it decreases from $0.2344$ to $0.1180$.  At $\nu=0.18$ the
residual is worse, so the gain is parameter dependent.

\begin{table}[H]
 \centering
 \small
 \caption{95\% confidence sets in the residual study.  Lengths are medians of
 hull lengths over 200 graphs; coverage and empty-set frequency are shown for
 the residual.  The full parameter range has length $1.125$.}
 \label{tab:residual-intervals}
 \begin{tabular}{r rr rr}
  \toprule
  &\multicolumn{2}{c}{Median length}&\multicolumn{2}{c}{Residual (\%)}\\
  \cmidrule(lr){2-3}\cmidrule(lr){4-5}
  $\nu$&Degree&Residual&Coverage&Empty\\
  \midrule
  0.18000 & 0.4303 & 0.4652 & 100.0 & 0.0 \\
  0.45125 & 0.9297 & 0.7425 & 100.0 & 0.0 \\
  0.50000 & 0.9138 & 0.7650 & 100.0 & 0.0 \\
  0.66125 & 0.8622 & 0.7847 & 99.5 & 0.0 \\
  0.78125 & 0.8050 & 0.6738 & 99.5 & 0.5 \\
  0.91125 & 0.7638 & 0.5997 & 98.0 & 0.5 \\
  \bottomrule
 \end{tabular}
\end{table}

\begin{figure}[H]
 \centering
 \includegraphics[width=0.91\textwidth]{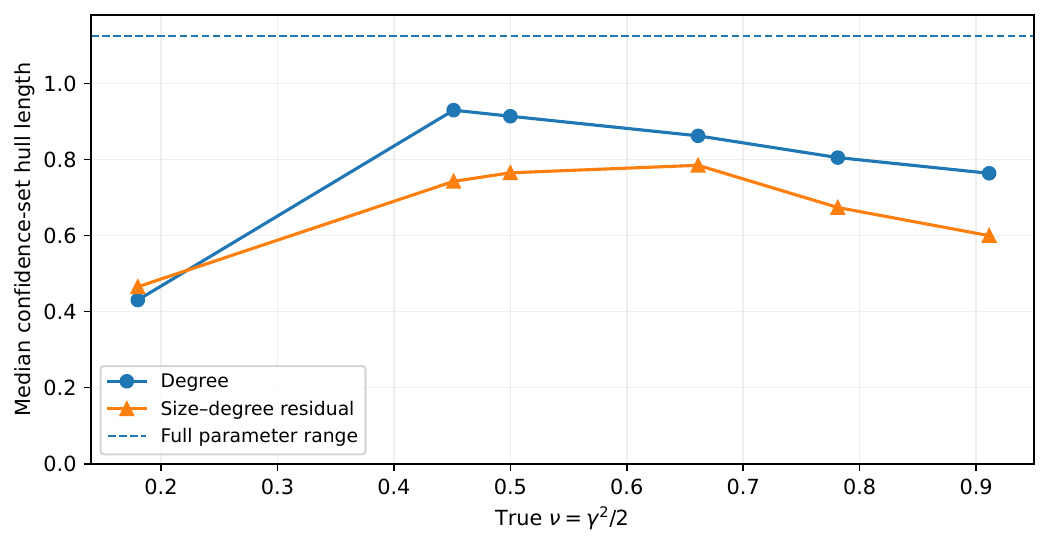}
 \caption{Median confidence-set hull length for degree and residual scores; lower is narrower.}
 \label{fig:residual-intervals}
\end{figure}

At $\nu=0.78125$ and $0.91125$, the residual median width decreases from
$0.8050$ to $0.6738$ and from $0.7638$ to $0.5997$, respectively.  Empirical
residual coverage ranges from $98\%$ to $100\%$ in the displayed groups.

\FloatBarrier
\subsection{Multiresolution diagnostics}
\label{subsec:numerical-paired}

The third study asks whether increasing $h$ improves the finite-resolution
approximation in three observable ways: covariance, expected edge count, and a
fixed residual decision.  For $\gamma^2<2$, the periodic pixel and continuous
expected edge counts are
\begin{align}
 \EE_{Q_{\gamma,n,h}^{\T}}E
 &=\frac{n^2}{2}\iint_{(\T^2)^2}
 k_n^{\T}(x,y)e^{\gamma^2K_h^{\T}(x,y)}\dd x\dd y,
 \label{eq:pixel-edge-mean}\\
 \shortintertext{and}
 \EE_{P_{\gamma,n}^{\T}}E
 &=\frac{n^2}{2}\iint_{(\T^2)^2}
 k_n^{\T}(x,y)e^{\gamma^2K_{\T}(x-y)}\dd x\dd y.
 \label{eq:continuum-edge-mean}
\end{align}
Let $D_h(\gamma,n)$ be one minus their ratio.

\begin{table}[!htbp]
 \centering
 \small
 \caption{Spectral and edge-mean diagnostics for the periodic sequence.  The
 two right columns are $100D_h(\gamma,2000)$ in percent; smaller values mean
 closer agreement of this second-order observable.}
 \label{tab:spectral-refinement}
 \begin{tabular}{rrrrr}
  \toprule
  $h$&$v_h-\log h$&$K_h^{\T}((1/4,0),0)$
       &$\gamma=1$&$\gamma=1.25$\\
  \midrule
  256  & $-0.985780$ & 0.177037 & 16.952 & 49.691\\
  512  & $-0.985794$ & 0.177026 &  8.402 & 36.623\\
  1024 & $-0.985798$ & 0.177023 &  4.143 & 26.954\\
  2048 & $-0.985799$ & 0.177022 &  2.052 & 19.867\\
  4096 & $-0.985799$ & 0.177022 &  1.020 & 14.658\\
  \bottomrule
 \end{tabular}
\end{table}

The covariance columns stabilize over the displayed resolutions and the
edge-mean deficits decrease.  The local doubling exponents between
$h=2048$ and $4096$ are $1.008$ for $\gamma=1$ and $0.439$ for
$\gamma=1.25$, close to $2-\gamma^2=1$ and $0.4375$.

For a score-level diagnostic, fix the $h=256$ residual pilot and calibration
library.  Use
\[
 B^{(1)}=[1,1.025]\times[2000,2125],
 \qquad
 B^{(1.25)}=[1.25,1.275]\times[2000,2125],
\]
and let $c_{B^{(\gamma)}}$ be the twentieth largest upper calibration bound.
Thus
\begin{equation}
 p_{B^{(\gamma)}}^\rho(G)\le0.05
 \quad\Longleftrightarrow\quad
 \rho_{B^{(\gamma)}}(G)>c_{B^{(\gamma)}}.
 \label{eq:frozen-decision}
\end{equation}
The score and threshold remain fixed across resolutions.  Common Gaussian
modes, coupled Cox counts and positions, and shared edge marks give paired
graphs with the required finite-model marginals.

\begin{table}[!htbp]
 \centering
 \small
 \caption{Paired stability of the frozen residual score at $n=2000$.
 Correlation is the Pearson correlation of paired scores; up/down are
 $0\to1$ and $1\to0$ threshold crossings.}
 \label{tab:paired-decisions}
 \begin{tabular}{crrrrr}
  \toprule
  $h\to h'$&$\gamma$&Pairs&Median $|\rho_{h'}-\rho_h|$
      &Correlation&Up/down\\
  \midrule
  256\,$\to$\,1024  & 1.00 & 600 & 0.108028 & 0.860628 & 0/0\\
  256\,$\to$\,1024  & 1.25 & 600 & 0.171825 & 0.829341 & 4/2\\
  1024\,$\to$\,2048 & 1.00 & 200 & 0.025867 & 0.962873 & 0/1\\
  1024\,$\to$\,2048 & 1.25 & 200 & 0.041658 & 0.977321 & 0/0\\
  \bottomrule
 \end{tabular}
\end{table}

The later comparison has smaller score gaps and higher correlations in both
groups.  Because the threshold is frozen, these are stability diagnostics,
not coverage calculations.

\FloatBarrier

\end{document}